\documentclass[12pt,leqno]{article}
\usepackage{amssymb,amsfonts,amsmath,amsthm,amscd,mathrsfs}
\usepackage{graphicx}
\def\IE{{\mathbb E}}
\def\IP{{\mathbb P}}
\def\IR{{\mathbb R}}

\def\IQ{{\mathbb Q}}
\def\IZ{{\mathbb Z}}

\def\n{\noindent}
\def\dsl{\textstyle\sum\limits}

\def\dis{\displaystyle}
\def\o{\omega}
\def\fr{\mbox{\footnotesize $\dis\frac{1}{2}$}}

\def\ve{\varepsilon}
\def\f{\footnotesize}
\def\r{\rightarrow}
\def\point{{\mbox{\large $.$}}}

\def\wt{\widetilde}

\def\cA{{\cal A}}
\def\cB{{\cal B}}
\def\cC{{\cal C}}

\def\cL{{\cal L}}

\def\cI{{\cal I}}
\def\cJ{{\cal J}}
\def\cK{{\cal K}}
\def\cH{{\cal H}}

\def\cF{{\cal F}}

\def\cG{{\cal G}}

\def\cU{{\cal U}}
\def\cV{{\cal V}}

\def\cZ{{\cal Z}}

\def\wtv{\widetilde{\varphi}}
\def\wte{\widetilde{E}}
\def\wtie{\widetilde{\IE}}

\dimendef\dimen=0

\newtheorem{theorem}{Theorem}[section]
\newtheorem{lemma}[theorem]{Lemma}
\newtheorem{corollary}[theorem]{Corollary}
\newtheorem{proposition}[theorem]{Proposition}
\newtheorem{remark}[theorem]{Remark}

\begin{document}

\noindent
~

\bigskip
\begin{center}
{\bf COUPLING AND AN APPLICATION TO LEVEL-SET \\
PERCOLATION OF THE GAUSSIAN FREE FIELD}
\end{center}

\begin{center}
Alain-Sol Sznitman
\end{center}


\bigskip\bigskip
\begin{abstract}
In the present article we consider a general enough set-up and obtain a refinement of the coupling between the Gaussian free field and random interlacements recently constructed by Titus Lupu in \cite{Lupu}. We apply our results to level-set percolation of the Gaussian free field on a $(d+1)$-regular tree, when $d \ge 2$, and derive bounds on the critical value $h_*$. In particular, we show that $0 < h_* < \sqrt{2u_*}$, where $u_*$ denotes the critical level for the percolation of the vacant set of random interlacements on a $(d+1)$-regular tree.
\end{abstract}

\vfill 

\n
Departement Mathematik 
\\
ETH Z\"urich\\
CH-8092 Z\"urich\\
Switzerland

\vfill

~
\newpage
\thispagestyle{empty}
~

\newpage
\setcounter{page}{1}

 \setcounter{section}{-1}
 \section{Introduction}
  \setcounter{equation}{0}
Cable processes constitute a potent tool in conjunction with Dynkin-type isomorphism theorems as shown in the recent articles \cite{Lupu}, \cite{Lupu15}, \cite{QianWern}, \cite{Zhai}. In the present work, in a general enough set-up, we obtain a refinement of the coupling between the Gaussian free field and random interlacements recently constructed in \cite{Lupu}. We apply our results to level-set percolation of the Gaussian free field on the $(d+1)$-regular tree $(d \ge 2)$ endowed with unit weights. We characterize the critical value $h_*$ for level-set percolation in terms of a certain variational problem and establish upper and lower bounds on $h_*$. In particular, we show that for all $d \ge 2$, $0 < h_* < \sqrt{2 u_*}$, where $u_*$ stands for the critical value for the percolation of the vacant set of random interlacements on the $(d+1)$-regular tree, an explicit quantity by the results of \cite{Teix09b}.
  
\medskip
We now describe our results in more detail. We consider a locally finite, connected, transient weighted graph, with vertex set $E$, and symmetric weights $c_{x,y} = c_{y, x} \ge 0$, which are positive exactly when $x \sim y$, that is, when $x$ and $y$ are neighbors. We consider the discrete time simple random walk on this weighted graph. When in $x \in E$, the walk jumps to a neighbor $y$ of $x$ with probability $c_{x,y} /\lambda_x$, where
\begin{equation}\label{0.1}
\lambda_x = \dsl_{x ' \sim x} c_{x,x'}, \; \mbox{for $x \in E$}.
\end{equation}

\n
We write $P_x$ for the law of the walk starting at $x, E_x$ for the corresponding expectation, and $(Z_k)_{k\ge 0}$, for the walk. The Green function is symmetric and equals
\begin{equation}\label{0.2}
g(x,y) = \mbox{\f $\dis\frac{1}{\lambda_y}$} \; E_x \Big[\dsl^\infty_{k=0} \,1\{Z_k = y\}\Big], \; \mbox{for $x,y \in E$}.
\end{equation}

\n
We denote by $\IP^G$ the canonical law on $\IR^E$ of the Gaussian free field on $E$, and by $(\varphi_x)_{x \in E}$ the canonical field, so that under $\IP^G$
\begin{equation}\label{0.3}
\mbox{$(\varphi_x)_{x \in E}$ is a centered Gaussian field with covariance $g(\cdot,\cdot)$}.
\end{equation}

\n
We also consider $u > 0$, and on some auxiliary probability spaces governed by the probability $\IP^I$ (see for instance \cite{Szni12b}), the non-negative field
\begin{equation}\label{0.4}
\mbox{$(\ell_{x,u})_{x \in E}$ of occupation-times of random interlacements at level $u$ on $E$,}
\end{equation}
as well as
\begin{equation}\label{0.5}
\cI^u = \{x \in E;\, \ell_{x,u} > 0\} \;\mbox{and} \; \cV^u = \{x \in E; \ell_{x,u} = 0\},
\end{equation}

\n
the respective interlacement set and vacant set at level $u$. A Dynkin-type isomorphism theorem is known to hold in this context (see Theorem 0.1 of \cite{Szni12b}, see also 
\cite{EiseKaspMarcRoseShi00}, \cite{MarcRose06}, \cite{SaboTarr})
\begin{equation}\label{0.6}
\begin{array}{l}
\big(\fr \;\varphi^2_x + \ell_{x,u}\big)_{x \in E} \;\mbox{under the product measure $\IP^G \otimes \IP^I$ has same law as}
\\[2ex]
\big(\fr \;(\varphi_x - \sqrt{2u})^2\big)_{x \in E} \;\mbox{under $\IP^G$}.
\end{array}
\end{equation}

\n
One can attach a cable system to the above weighted graph, see Section 1. It has vertex set $\wt{E} \supseteq E$ such that all edges $\{x,y\}$ in $E$ are linked by a compact segment of length $(2c_{x,y})^{-1}$. One defines on the cable system a continuous diffusion behaving as standard Brownian motion in the interior of each such segment. It has a continuous symmetric Green function $\wt{g}(z,z')$, $z,z' \in \wt{E}$, with respect to the Lebesgue measure on $\wt{E}$, which extends the Green function $g(\cdot,\cdot)$ in (\ref{0.2}) of the discrete time walk. One can then consider the Gaussian free field on $\wt{E}$, which is a continuous centered Gaussian field $(\wtv_z)_{z \in \wt{E}}$ with covariance $\wt{g}(\cdot,\cdot)$. Under the assumption that the sign clusters of the Gaussian free field on $\wt{E}$ are a.s. bounded, it was shown in Theorem 3 of \cite{Lupu} that one can construct a coupling of $(\eta_x)_{x \in E}$, a Gaussian free field on $E$, and $\cV^u$ in (\ref{0.5}) such that
\begin{equation}\label{0.7}
\cV^u \supseteq \{\eta > \sqrt{2u}\} \big(\stackrel{\rm def}{=}  \{x \in E; \,\eta_x > \sqrt{2u}\}\big), \;\mbox{a.s.}\,.
\end{equation}
The assumption on the sign clusters of the Gaussian free field on $\wt{E}$ holds for instance in the case of $\IZ^d$, $d \ge 3$, endowed with unit weights, as shown in \cite{Lupu} (we will see in Section 4 that it holds as well for the $(d+1)$-regular tree, $d \ge 2$, endowed with unit weights).

\medskip
In the present work we refine the construction of \cite{Lupu}, and under the (mild) additional assumption that $\sup_{x \in E} g(x,x) < \infty$ (see also (\ref{1.40a}) or (\ref{1.40b}) for a weakening of this assumption), we construct a coupling of $(\eta_x)_{x \in E}$, $(\varphi_x)_{x \in E}$ Gaussian free fields on $E$, and $\cV^u$ as in (\ref{0.5}), such that
\begin{equation}\label{0.8}
\begin{array}{rl}
{\rm i)} & \mbox{$\varphi$ and $\cV^u$ are independent,}
\\[1ex]
{\rm ii)} & \mbox{a.s., for any $A \subseteq (0,\infty), \cV^u \cap \{ \varphi \in A\}\ \supseteq \{ \eta \in \sqrt{2u} + A\}$}
\end{array}
\end{equation}

\n
(thus, choosing $A = (0,\infty)$, ii) above clearly refines (\ref{0.7})).

\medskip
We prove (\ref{0.8}) in Corollary \ref{cor2.4} (see also Remark \ref{rem2.6}). It is a direct consequence of a coupling between the Gaussian free field on the cable system and the random interlacements at level $u$ on the cable system that we construct in Section 2. It refines the coupling in \cite{Lupu}, in essence, through the use of the strong Markov property of the Gaussian free field on the cable system. Our main statement appears in Theorem \ref{theo2.3}.

\medskip
We then consider in Sections 3 and 4 an application to level-set percolation of the Gaussian free field on the $(d+1)$-regular tree $T$ endowed with unit weights (with $d \ge 2$). We are interested by the percolative property of $\{\varphi \ge h\}$, for $h$ in $\IR$, with $(\varphi_x)_{x \in T}$ the Gaussian free field on $T$. We introduce the quantity (see (\ref{3.17}) for its spectral interpretation)
\begin{equation}\label{0.9}
\lambda_h = \sup\Big\{\mbox{\f $\dis\frac{d^2}{\sqrt{d^2 - 1}}$} \dis\int^\infty_h \,\dis\int^\infty_h e^{-\frac{a^2}{2d} + ab - \frac{b^2}{2d}} f(a) \,f(b) \,\nu(da)\, \nu(db); \; \dis\int_\IR f^2 (a) \,\nu(da) = 1\Big\},
\end{equation}
where $\nu$ is the centered Gaussian law
\begin{equation}\label{0.10}
\nu(da) = \mbox{\f $\dis\frac{1}{\sqrt{2 \pi \sigma}}$} \;e^{-\frac{a^2}{2\sigma^2}} \,da, \;\mbox{with} \; \sigma^2 \stackrel{\rm def}{=} \; \mbox{\f $\dis\frac{d}{d^2 - 1}$} .
\end{equation}
We show in Proposition \ref{prop3.1} that
\begin{equation}\label{0.11}
\mbox{$h \r \lambda_h$ is a continuous decreasing bijection between $\IR$ and $(0,d)$,}
\end{equation}

\n
and identify in Proposition \ref{prop3.3} the critical value $h_*$ for level-set percolation of $\varphi$:
\begin{equation}\label{0.12}
\mbox{$h_*$ is the unique value such that $\lambda_{h_*} = 1$}
\end{equation}

\n
(so $\{\varphi \ge h\}$ only has finite components when $h > h_*$, and has an infinite component when $h < h_*$).

\medskip
Level-set percolation of $(\varphi_x)_{x \in T}$ can be recast in terms of the study of a branching Markov chain with a barrier, where the Markov chain under consideration corresponds to a suitable Ornstein-Uhlenbeck transition operator, see Section 3. This is very much in the spirit of some of the models discussed in \cite{BiggKypr04}, and indeed, we employ methods of the theory of branching processes in Section 3, which in turn require gaining information on various spectral objects underlying $\lambda_h$ from (\ref{0.9}).

\medskip
Unlike percolation for the vacant set $\cV^u$ of random interlacements, where the critical value $u_*$ is explicit in the sense that $d \exp\{-u_* \frac{(d-1)}{d}^2\} = 1$, see \cite{Teix09b}, there does not seem to be a closed formula for $h_*$, or for $\lambda_h$. Our main results in Section 4 provide bounds on $\lambda_h$ and $h_*$. We show in Theorem \ref{theo4.3} that
\begin{equation}\label{0.13}
\lambda_h > d \overline{\Phi}\big(h \mbox{\f $\dis\frac{(d-1)}{\sqrt{d}}$}\big), \; \mbox{for $h \in \IR$, with $\overline{\Phi} (a) = \mbox{\f $\dis\frac{1}{\sqrt{2 \pi}}$}\;\dis\int^\infty_a e^{-\frac{t^2}{2}} \;dt$}
\end{equation}
(so $\lambda_0 > \frac{d}{2}$), and that

\vspace{-3ex}
\begin{equation}\label{0.14}
\lambda_h \le \lambda_0 \,e^{-\frac{h^2}{2} \;\frac{(d-1)^2}{d}}, \; \mbox{for $h \ge 0$}.
\end{equation}

\n
The upper bound (\ref{0.14}) comes as an application of the coupling (\ref{0.8}) (and we do not know of a direct proof). As a result of (\ref{0.12}) - (\ref{0.14}) we bound $h_*$ in Corollary \ref{cor4.5}:
\begin{equation}\label{0.15}
\begin{array}{c}
0 \le h_\Delta < h_* \le h_\square < \sqrt{2 u_*}, \;\mbox{with}
\\[1ex]
 d \overline{\Phi} \big(h_\Delta\, \mbox{\f $\dis\frac{(d-1)}{\sqrt{d}}$}\big) = 1, \; \lambda_0\, e^{-\frac{h^2_\square (d-1)^2}{2d}} = 1, \;\mbox{and}\;  d \,e^{-u_* \,\frac{(d-1)^2}{d}} =1.  
\end{array}
\end{equation}

\n
In particular, we prove that $0 < h_* < \sqrt{2u_*}$. One can naturally wonder how general this inequality is, see also Remark \ref{rem4.6}.

\medskip
We now explain how this article is organized. In Section 1, we collect useful information on the cable system, on the strong Markov property of the Gaussian free field on the cable system, and on the coupling from \cite{Lupu}. In Section 2, we construct our main coupling in Theorem \ref{theo2.3}. The application (\ref{0.8}) is shown in Corollary \ref{cor2.4}, see also Remark \ref{rem2.6}. Section 3 brings into place the set-up for the case of regular trees. The claims (\ref{0.11}) and (\ref{0.12}) respectively appear in Proposition \ref{prop3.1} and \ref{prop3.3}. In Section 4, we derive the main bounds (\ref{0.13}), (\ref{0.14}) on $\lambda_h$ in Theorem \ref{theo4.3}, and the resulting bounds on $h_*$ in Corollary \ref{cor4.5}. 

\bigskip\n
{\bf Acknowledgments:} We wish to thank Titus Lupu for useful discussions.

 \section{Cable systems: some preliminaries}
 \setcounter{equation}{0}

In this section we collect some facts concerning various objects attached to the cable system. In particular, we discuss the strong Markov property of the Gaussian free field on the cable system, as well as some features of random interlacements and the isomorphism theorem on the cable system. 

\medskip
We keep the notation from the introduction. Given our basic locally finite, connected, transient weighted graph with vertex set $E$, the cable system (or metric graph) is obtained by attaching to each edge $e = \{x,y\}$ of the above graph a compact interval with length $(2 c_{x,y})^{-1}$ with endpoints respectively identified to $x$ and $y$. We denote by $\wt{E} \supseteq E$ the resulting set obtained by glueing the above intervals so that $\wt{E} \backslash E$ is the disjoint union of the sets $I_e$, so that for each edge $e$, $I_e$ is homeomorphic to the interior of the interval attached to $e$. We denote by $m$ the Lebesgue measure on $\wt{E}$. For further details, we also refer to Section 2 of \cite{Lupu}, Section 2 of \cite{Folz14}, and Section 1 of \cite{Varo85}, as well as to Section 3 of \cite{Zhai}, which discusses a discrete space approximation to the cable system.

\medskip
We also denote by $d(\cdot,\cdot)$ the (geodesic) distance on $\wt{E}$ for which we attach length 1 to each $I_e$ (instead of $(2 c_{x,y})^{-1}$, with $e = \{x,y\}$), so that the restriction of $d(\cdot,\cdot)$ to $E \times E$ coincides with the graph distance on $E$. For $x \in E$, $N \ge 1$, we set
\begin{equation}\label{1.1}
\begin{split}
B_N(x) & = \{z \in \wt{E}; \,d(x,z) \le N\}, \; B^{\circ}_N(x) = \{x \in \wt{E}; \, d(x,z) < N\}, \;\mbox{and}
\\
S_N(x) & = \{z \in \wt{E}; \,d(x,z) = N\} = \{y \in E; \, d(x,y) = N\}.
\end{split}
\end{equation}

\n
On $\wt{E}$ one can define a continuous diffusion, via probabilities $\wt{P}_z, z \in \wt{E}$, governing $X$ the canonical process with possibly finite life on $\wt{E}$, so that on each $I_e$, $X$ behaves as a standard Brownian motion. This diffusion has continuous space-time local times with respect to the Lebesgue measure $m$ on $\wt{E}$, and when in $x \in E$, reaches a neighboring site in $E$ after a local time at $x$, which is exponential with parameter $\lambda_x$, see (\ref{0.1}), and equal to $y \sim x$ (among all neighbors of $x$) with probability $c_{x,y}/\lambda_x$, see Section 2 of \cite{Lupu}.

\medskip
Given $U$ open subset of $\wt{E}$ and $B$ closed subset of $\wt{E}$, we denote by $T_U = \inf\{s \ge 0; X_s \notin U\}$ and $H_B = \inf\{s \ge 0; X_s \in B\}$ the respective exit time from $U$ and entrance time in $B$ of $X$. We will also occasionally consider the case when $B$ is open (for instance in Lemma \ref{lem1.4}). Given $U \subseteq \wt{E}$ open, we denote by $\wt{g}_U(\cdot,\cdot)$ the Green function with respect to the measure $m$ of the diffusion on the cable killed when exiting $U$, so that
\begin{equation}\label{1.2}
\mbox{$\wt{g}_U(z,z')$ is continuous, symmetric, and vanishes if $z$ or $z'$ is not in $U$,}
\end{equation}

\n
and for all $z \in \wt{E}$, $\wt{g}_U(z,\cdot)$ ($= \wt{g}_U(\cdot,z)$) is harmonic on $U \backslash \{z\}$. When $U = \wt{E}$, one recovers the Green function $\wt{g}(\cdot,\cdot)$ mentioned in the introduction. It has the property (see (\ref{0.2}) for notation)
\begin{equation}\label{1.3}
\wt{g}(x,y) = g(x,y), \;\mbox{for $x,y \in E$}.
\end{equation}

\n
The killed Green function also has the monotone convergence property: when $U_n \uparrow U$,
\begin{equation}\label{1.4}
\wt{g}_U (z,z') = \lim\limits_n \uparrow \wt{g}_{U_n}(z,z'), \;\mbox{for $z,z' \in \wt{E}$}
\end{equation}

\n
(for instance this follows from monotone convergence for resolvents combined with the continuity and harmonicity of $\wt{g}_{U_n}(z,\cdot)$ and $\wt{g}_U(z,\cdot)$).

\medskip
We now turn to the Gaussian free field on the cable system $\wt{E}$. On the canonical space $\wt{\Omega}$ of continuous real-valued functions on $\wt{E}$ endowed with the canonical $\sigma$-algebra $\cA$ generated by the canonical coordinates $\wt{\varphi}_z$ (we also sometimes write $\wt{\varphi}(z)$), $z \in \wt{E}$, we consider the probability $\wt{\IP}^G$ such that
\begin{equation}\label{1.5}
\begin{array}{l}
\mbox{under $\wt{\IP}^G$, $(\wt{\varphi}_z)_{z \in \wt{E}}$ is a centered Gaussian process with covariance}
\\[0.3ex]
\wt{\IE}^G[\wt{\varphi}_z \wt{\varphi}_{z'}] = \wt{g}(z,z') \;\mbox{(where $\wt{\IE}^G$ denotes $\wt{\IP}^G$-expectation}).
\end{array}
\end{equation}
By (\ref{0.3}) and (\ref{1.5}), we see that
\begin{equation}\label{1.6}
\mbox{the law of $(\wt{\varphi}_x)_{x \in E}$ under $\wt{\IP}^G$ is equal to $\IP^G$}.
\end{equation}

\n
We first state the Markov property of $\wt{\varphi}$. We consider $K \subseteq \wt{E}$ a compact subset with finitely many connected components (note that $\partial K$ is finite). For $U = \wt{E} \backslash K$ and $z \in \wt{E}$, we set (with $T_U$ the exit time of $X$ from $U$)
\begin{equation}\label{1.7}
h_U(z) = \wt{E}_z [\wt{\varphi}(X_{T_U}), \, T_U < \infty], \, z \in \wt{E}.
\end{equation}

\n
This function is continuous, coincides with $\varphi$ on $K$, and tends to $0$ at infinity. It is also harmonic on $U$. Then $\wt{\varphi}_z - h_U(z)$, $z \in \wt{E}$, is a continuous function, which vanishes on $K$ and the Markov property states that
\begin{equation}\label{1.8}
\begin{array}{l}
\mbox{$\big(\wt{\varphi}_z - h_U(z)\big)_{z \in \wt{E}}$ is a centered Gaussian field with covariance $\wt{g}_U(\cdot,\cdot)$}
\\
\mbox{independent of $\cA_K^+$, where $\cA^+_K$ is the $\sigma$-algebra}
\end{array}
\end{equation}

\vspace{-2ex}
\begin{equation}\label{1.9}
\cA^+_K = \bigcap_{\varepsilon > 0} \, \sigma (\wt{\varphi}_z, z \in K^\ve)
\end{equation}

\smallskip\n
with $K^\ve$ the open $\ve$-neighborhood of $K$ in the $d(\cdot,\cdot)$-distance, see above (\ref{1.1}). Incidentally, for $F \subseteq \wt{E}$, we will write $\cA_F = \sigma(\wt{\varphi}_z, z \in F)$ ($\subseteq \cA = \cA_{\wt{E}}$).

\medskip
We now turn to the strong Markov property. We say that $\cK$ is a compatible random compact subset of $\wt{E}$ (or as a shorthand that $\cK$ is compatible, see chapter 2 \S 3 of \cite{Roza82} for the terminology), if for each $\o \in \wt{\Omega}$, $\cK(\o) \subseteq \wt{E}$ is a compact subset with finitely many components, which is the closure of its interior, and such that for any open subset $O \subseteq \wt{E}$
\begin{equation}\label{1.10}
\{\cK \subseteq O\} (\stackrel{\rm def}{=} \{\o \in \wt{\Omega}; \, \cK(\o) \subseteq O\}) \in \cA_O \;\mbox{(see below (\ref{1.9}) for the notation)}.
\end{equation}

\n
We attach to $\cK$ the $\sigma$-algebra, see Theorem 2, p.~89 of \cite{Roza82}
\begin{align}
\cA^+_\cK = \{ &A \in \cA; \,A \cap \{\cK \subseteq K\} \in \cA^+_K, \;\mbox{for all $K \subseteq E$, which is} \label{1.11}
\\
&\mbox{compact and the closure of its interior$\}$}. \nonumber
\end{align}

\n
The lemma below and the subsequent remark collect some useful properties of $\cA^+_\cK$.

\begin{lemma}\label{lem1.1}
Let $\cK$ be compatible and define the $\sigma$-algebra $\cG_\cK = \sigma(\{\cK \subseteq O\}$; $O$ open subset of $\wt{E})$. Then, one has
\begin{equation}\label{1.12}
\left\{ \begin{array}{rll}
{\rm i)} & \cG_\cK \subseteq \cA^+_{\cK},
\\[1ex]
{\rm ii)} & \mbox{for each $z \in \wt{E}$}, & \{z \in \cK\} \in \cG_\cK \,(\subseteq \cA^+_\cK) ,
\\[1ex]
{\rm iii)} & \mbox{for each $z \in \wt{E}$}, & \mbox{$\wt{\varphi}_z \,1\{z \in \cK\}$ is $\cA^+_\cK$-measurable~.}
\end{array}\right.
\end{equation}
\end{lemma}

\begin{proof}
To prove i), we note that for any $K$ as in (\ref{1.11}) and $O$ open in $\wt{E}$,
\begin{equation}\label{1.13a}
\{\cK \subseteq O\} \cap \{\cK \subseteq K\} = \bigcap\limits_{\ve > 0} \,\{ \cK \subseteq O \cap K^\ve\} \stackrel{(\ref{1.10})}{\in} \cA^+_K,
\end{equation}

\smallskip\n
whence $\{\cK \subseteq O\} \in \cA^+_\cK$ and i) follows. To prove ii) we choose $O = \{z\}^c$ and apply i).

\pagebreak
As for iii), consider $z \in \wt{E}$, and $K$ as in (\ref{1.11}). Then, $\wt{\varphi}_z \,1\{z \in \cK, \cK \subseteq K\}$ vanishes identically on $\wt{\Omega}$ if $z \notin K$, and if $z \in K$, both $\wt{\varphi}_z$ and $1 \{z \in \cK, \cK \subseteq K\}$ are $\cA^+_K$-measurable by ii) and (\ref{1.13a}). The claim iii) follows.
\end{proof}

\begin{remark}\label{rem1.2} \rm ~

\medskip\n
1) If $\cK$ and $\cL$ are compatible random compact subsets such that $\cK \subseteq \cL$, then 
\begin{equation}\label{1.13}
\cA^+_\cK \subseteq \cA^+_\cL
\end{equation}
(indeed, for $A \in \cA^+_\cK$ and $K$ as in (\ref{1.11}), $A \cap \{\cL \subseteq K\} = A \cap \{\cK \subseteq K\} \cap \{\cL \subseteq K\} \in \cA^+_K\}$).

\bigskip\n
2) As described in Chapter 2 \S 4 of \cite{Roza82}, one can approximate a compatible compact subset $\cK$ from above as follows. One chops (in a dyadic procedure) each compact segment attached to an edge of $E$ into $2^n$ closed segments of equal length, and defines $\cK^n$ as the union of the finitely many such segments intersecting the interior of $\cK$. In this fashion $\cK^n$ takes values in a countable set (of possible ``shapes''), and
\begin{equation}\label{1.14}
\left\{ \begin{array}{rll}
{\rm i)} &\cK^n \downarrow \cK,
\\[1ex]
{\rm ii)} & \mbox{$\cK^n$ is compatible for each $n$,} 
\\[1ex]
{\rm iii)} &\cA^+_{\cK^n} \downarrow \cA^+_\cK
\end{array}\right.
\end{equation}

\medskip\n
(in the case of iii) simply note that for $A \in \bigcap_n \cA^+_{\cK^n}$ and $K$ as in (\ref{1.11}), one has for $\ve > 0$, $A \cap \{\cK \subseteq K^\ve\} = \bigcup_n A \cap \{\cK^n \subseteq K^\ve\}$ which belongs to $\cA_{K^\ve}$ as a straightforward consequence of ii), so that $A \cap \{\cK \subseteq K\} = \bigcap_{m \ge 1} A \cap \{\cK \subseteq K^{1/m}\} \in \cA^+_K$, and iii) follows). Actually, one also has $\cG_{\cK^n} \uparrow \cG_\cK$, see \cite{Roza82}, p.~87, but we will not need this fact.

\bigskip\n
3) If $\cK$ is a compatible random compact subset of $\wt{E}$, and $\cK^n$ defined as above, we set 
\begin{equation*}
\mbox{$\cU = \wt{E} \backslash \cK$ and $\cU^n = \wt{E} \backslash \cK^n$, so that $\cU^n \uparrow \cU$}.
\end{equation*}

\n
Then, for each $n \ge 0$, $z,z' \in \wt{E}$, by (\ref{1.14}) ii) (and using also (\ref{1.12}) iii) for $h_{\cU^n}(z)$)
\begin{equation}\label{1.15}
\mbox{$h_{\cU^n}(z)$ and $\wt{g}_{\cU^n} (z,z')$ are $\cA^+_{\cK^n}$-measurable},
\end{equation}

\n
(actually, $\wt{g}_{\cU^n}(z,z')$ is even $\cG_{\cK^n}$-measurable).

\medskip
Moreover, by dominated convergence, cf.~(\ref{1.7}), and the monotone convergence property (\ref{1.4})
\begin{equation}\label{1.16}
h_{\cU}(z) = \lim\limits_n h_{\cU^n}(z) \;\mbox{and} \; \wt{g}_\cU (z,z') = \lim\limits_n \uparrow \wt{g}_{\cU^n} (z,z'),
\end{equation}
and by (\ref{1.14}) iii) we see that
\begin{equation}\label{1.17}
\mbox{$h_\cU(z)$ and $\wt{g}_\cU(z,z')$ are $\cA^+_\cK$-measurable}
\end{equation}

\smallskip\n
(and respectively continuous in $z$, and $z,z'$). \hfill $\square$
\end{remark}

\medskip
We can now state the {\it strong Markov property} of $\wt{\varphi}$, see in particular Theorem 4, p.~92 of \cite{Roza82}. When $\cK$ is a compatible random compact subset of $\wt{E}$,
\begin{align}
&\mbox{under $\wt{\IP}^G$, conditionally on $\cA^+_{\cK}$, $(\wt{\varphi}_z)_{z \in \wt{E}}$ is a Gaussian field} \label{1.18}
\\
&\mbox{with mean $\big(h_\cU(z)\big)_{z \in \wt{E}}$ and covariance $\wt{g}_\cU(\cdot,\cdot)$}. \nonumber
\end{align}

\n
We now turn to the discussion of random interlacements on $\wt{E}$. For $C$ finite subset of $E$, we denote by ${\rm cap}(C)$ the capacity of $C$ and by $e_C$ the equilibrium measure of $C$ so that
\begin{equation}\label{1.19}
\begin{array}{l}
\mbox{${\rm cap}(C) = \sup\{E(\mu,\mu)^{-1}$; $\mu$ probability supported by $C\}$, where}
\\
E(\mu,\mu) = \fr \;\dsl_{x,y \in E} \mu(x) \,g(x,y)\, \mu(y),
\end{array}
\end{equation}
and one knows that when $C$ is not empty, the normalized equilibrium measure $e_C/{\rm cap}(C)$ is the unique maximizer in (\ref{1.19}).

\medskip
We consider a given level

\vspace{-5ex}
\begin{equation}\label{1.20}
u > 0
\end{equation}

\n
and on some suitable probability space $(\wt{W}, \cB, \wt{\IP}^I)$ a continuous non-negative random field $(\wt{\ell}_{z,\o})_{z \in \wt{E}}$ describing the field of local times with respect to the Lebesgue measure $m$ on $\wt{E}$ of random interlacements at level $u$ on $\wt{E}$, see Section 6 of \cite{Lupu}. If $x \in E$, and $B_N(x)$ is defined as in (\ref{1.1}), then $(\wt{\ell}_{z,u})_{z \in B_N(x)}$ is distributed as the restriction to $B_N(x)$ of the local time of a Poisson cloud of trajectories on $\wt{E}$ with intensity $u \,\wt{P}_{e_C}$, where $C = B_N(x) \cap E$, and $\wt{P}_{e_C} = \sum_{y \in E} e_C(y) \wt{P}_y$. Moreover (after discarding a negligible set), we can assume that
\begin{align}
&\wt{\cI}^u = \{z \in \wt{E}; \,\wt{\ell}_{z,u} > 0\} \;\mbox{is an open set which has only  unbounded}   \label{1.21}
\\
&\mbox{connected components,} \nonumber
\\[1ex]
&\partial \wt{\cI}^u \cap E = \phi, \label{1.22}
\end{align}

\n
where for $A \subseteq \wt{E}$, $\partial A$ denotes the boundary of $A$.

\medskip
It also readily follows that the restriction of $\wt{\ell}_{\point,u}$ to $E$ is distributed as the field of occupation times of random interlacements at level $u$ on $E$, cf.~(\ref{0.4}):
\begin{equation}\label{1.23}
\mbox{$(\wt{\ell}_{x,u})_{x \in E}$ under $\wt{\IP}^I$ has same law as $(\ell_{x,u})_{x \in E}$ under $\IP^I$}.
\end{equation}
In particular, cf.~(\ref{0.5}),
\begin{equation}\label{1.24}
\mbox{$\wt{\cI}^u \cap E$ under $\wt{\IP}^I$ has the same law as $\cI^u$ under $\IP^I$},
\end{equation}
so that for any finite subset $C$ of $E$
\begin{equation}\label{1.25}
\wt{\IP}^I [\wt{\cI}^u \cap C = \phi] = \IP^I[\cI^u \cap C = \phi] = e^{-u \,{\rm cap}(C)}.
\end{equation}

\medskip\n
As in Proposition 6.3 and Theorem 3 of \cite{Lupu}, one can construct some extension $(\Sigma, \cF, \wt{\IQ})$ of the product space $(\wt{\Omega} \times \wt{W}, \cA \otimes \cB, \wt{\IP}^G \otimes \wt{\IP}^I)$ endowed with the continuous fields $\wt{\gamma} = (\wt{\gamma}_z)_{z \in \wt{E}}$, $\wt{\varphi} = (\wt{\varphi}_z)_{z \in \wt{E}}$ and the non-negative continuous field $\wt{\ell} = (\wt{\ell}_{z,u})_{z \in \wt{E}}$ such that
\begin{align}
& \mbox{$(\wt{\varphi},\wt{\ell})$ under $\wt{\IQ}$ has the same law as under $\wt{\IP}^G \otimes \wt{\IP}^I$} \label{1.26}
\\[1ex]
& \mbox{both (\ref{1.21}) and (\ref{1.22}) hold} \label{1.27}
\\[1ex]
&\mbox{$\wt{\gamma}$ has the same law as $\wt{\varphi}$} \label{1.28}
\\[1ex]
& \mbox{$\wt{\IQ}$-a.s., for all $z \in \wt{E}$}, \;\fr \,(\wt{\gamma}_z - \sqrt{2u})^2 = \fr \, \wt{\varphi}^2_z + \wt{\ell}_{z,u}\,. \label{1.29}
\end{align}

\n
This coupling due to \cite{Lupu} sharpens (\ref{0.6}) and by (\ref{1.29})
\begin{equation}\label{1.30}
\mbox{$\wt{\IQ}$-a.s., $\wt{\gamma} - \sqrt{2u}$ does not vanish on $\wt{\cI}^u$ ($= \{z \in \wt{E}; \wt{\ell}_{z,u} > 0\}$).}
\end{equation}

\n
From now on, we will make the following assumption on the Gaussian free field on $\wt{E}$:
\begin{equation}\label{1.31}
\mbox{$\wt{\IP}^G$-a.s., $\{z \in \wt{E}; \,\wt{\varphi}_z > 0\}$ only has bounded connected components.}
\end{equation}

\begin{remark}\label{rem1.3} \rm
As shown in Section 5 of \cite{Lupu}, (\ref{1.31}) holds in particular in the case of $\IZ^d$, $d \ge 3$, endowed with unit weights. We will also see in Proposition \ref{prop4.1} below that (\ref{1.31}) holds as well when $E$ is the $(d+1)$-regular tree endowed with unit weights, and $d \ge 2$. \hfill $\square$
\end{remark}

Observe that $\wt{\cI}^u$ only has unbounded connected components and that $\wt{\gamma} - \sqrt{2u}$ does not vanish on $\wt{\cI}^u$. On the other hand, due to (\ref{1.31}), (\ref{1.28}), a.s.~$\{\wt{\gamma} > \sqrt{2u}\}$ only has bounded components. As a result, we see that
\begin{equation}\label{1.32}
\mbox{$\wt{\IQ}$-a.s., $\wt{\cI}^u \subseteq \{ \wt{\gamma} < \sqrt{2u}\}$ ($= \{z \in \wt{E}; \wt{\gamma}_z < \sqrt{2u}\}$)}.
\end{equation}

\n
It will also be technically convenient to make the (mild) assumption that
\begin{equation}\label{1.33}
\sup\limits_{x \in E} g(x,x) = g_0 < \infty \,.
\end{equation}

\n
This assumption will be used in the proof of the next lemma, as well as in the proofs of Proposition \ref{prop2.1} and Theorem \ref{theo2.3} in the next section. A useful but slightly more technical generalization of (\ref{1.33}) appears in (\ref{1.40a}) (or equivalently (\ref{1.40b})) of Remark \ref{rem1.5} below. We recall that $H_B$ stands for the entrance time of $X$ in $B$, see above (\ref{1.2}).

\begin{lemma}\label{lem1.4}
\begin{equation}\label{1.34}
\mbox{$\wt{\IP}^I$-a.s., for all $z \in \wt{E}$, $\wt{P}_z [H_{\wt{\cI}^u} < \infty] = 1$}.
\end{equation}
\end{lemma}

\begin{proof}
It clearly suffices to treat the case of $z \in E$. Since for $x \in E$, the successive visits of $X$ to distinct sites of $E$ under $\wt{P}_x$ are distributed as the discrete time walk $(Z_k)_{k \ge 0}$ under $P_x$ (see below (\ref{0.1}) for notation), and $\wt{\cI}^u \cap E$ is distributed as $\cI^u$, our claim will follow once we show that
\begin{equation}\label{1.35}
\mbox{$\IP^I$-a.s., for all $x \in E, P_x[H_{\cI^u} < \infty] = 1$},  
\end{equation}

\medskip\n
where for $B \subseteq E$, $H_B = \inf\{ k \ge 0; Z_k \in B\}$ is the entrance time of $Z$ in $B$. To this end, we will show that for all $x \in E$,
\begin{equation}\label{1.36}
\mbox{$P_x$-a.s., ${\rm cap}(\{Z_0,Z_1,\dots,Z_k\}) \underset{k}{\longrightarrow} \infty$}.
\end{equation}

\n
The claim (\ref{1.35}) will readily follow since for all $k \ge 0$,
\begin{align*}
\IP^I \otimes P_x[H_{\cI^u} = \infty] & = E_x \big[\IP^{\cI} [ \cI^u \cap \{Z_0,\dots, Z_k, \dots\} = \phi]\big]
\\
&\!\!\! \stackrel{(\ref{1.25})}{\le} E_x [\exp\{- u \,{\rm cap}(\{Z_0,\dots,Z_k\})\}] \underset{k}{\stackrel{(\ref{1.36})}{\longrightarrow}} 0\,.
\end{align*}
We thus prove (\ref{1.36}), and for this purpose first observe that for $y,y' \in E$
\begin{equation}\label{1.37}
\mbox{$P_{y'}$-a.s., $g(Z_k,y) \underset{k}{\longrightarrow} 0$}.
\end{equation}
Indeed, for any $k_0 \ge 1$ and $k \ge k_0$, $P_{y'}$-a.s.,
\begin{equation}\label{1.38}
\begin{split}
\mbox{\f $\dis\frac{g(Z_k,y)}{g(y,y)}$} & = P_{Z_k} [H_{\{y\}} < \infty] = P_{y'} \;[\mbox{after time $k$, $Z$ visits $y | \sigma(Z_0,\dots,Z_k)]$}
\\[1ex]
& \le \mbox{$P_{y'} [$ after time $k_0$, $Z$ visits $y | \sigma (Z_0,\dots,Z_k)]$}
\\[1ex]
&\!\!\! \underset{k \rightarrow \infty}{\longrightarrow} 1\{\mbox{after time $k_0$, $Z$ visits $y\}$ (by martingale convergence)}.
\end{split}
\end{equation}

\n
By transience, $P_{y'}$-a.s. for large $k_0$, the indicator function on the last line of (\ref{1.38}) vanishes, and (\ref{1.37}) follows.

\medskip
We will now deduce (\ref{1.36}). We construct by induction an increasing sequence of a.s. finite stopping times, via $S_0 = 0$ and
\begin{equation}\label{1.39}
S_{\ell + 1} = \inf\big\{k > S_\ell; \dsl_{1 \le j \le \ell} g(Z_k, Z_{S_j}) \le \mbox{\f $\dis\frac{g_0}{2}$}\big\} \;\mbox{for $\ell \ge 0$, with $g_0$ as in (\ref{1.33})}.
\end{equation}

\n
We now set $\mu = \frac{1}{\ell} \sum^\ell_{i=1} \delta_{Z_{S_i}}$. We know by (\ref{1.19}) that
\begin{equation}\label{1.40}
\begin{array}{l}
{\rm cap}(\{Z_0,\dots,Z_{S_\ell}\}) \ge E (\mu,\mu)^{-1}, \;\mbox{where}
\\[1ex]
E(\mu,\mu) = \mbox{\f $\dis\frac{1}{\ell^2}$} \;\dsl^\ell_{i,j = 1} g(Z_{S_i}, Z_{S_\ell}) \stackrel{(\ref{1.33})}{\le} \mbox{\f $\dis\frac{g_0}{\ell}$} + \mbox{\f $\dis\frac{2}{\ell^2}$}\;\dsl_{j < i} \,g(Z_{S_i},Z_{S_j})
\\[2ex]
\qquad \quad \stackrel{(\ref{1.39})}{\le}  \mbox{\f $\dis\frac{2 g_0}{\ell}$}.
\end{array}
\end{equation}

\medskip\n
This bound once inserted in (\ref{1.40}) shows that the capacity in the left member of (\ref{1.40}) is at least $\ell/(2 g_0)$ and (\ref{1.36}) follows. This completes the proof of Lemma \ref{lem1.4}.
\end{proof}

\begin{remark}\label{rem1.5} \rm Let us point out a slightly more technical but useful generalization of (\ref{1.33}). Namely, this condition is the existence of an increasing sequence of  bounded open connected subsets $D^\circ_N$ of $\wt{E}$, increasing to $\wt{E}$, containing a base point $x_0$ of $E$, with boundaries $\Delta_N = \partial D^\circ_N$ contained in $E$, so that $D^\circ_N$ is the connected component of $x_0$ in $\wt{E} \backslash \Delta_N$, $D^\circ_{N+1} \supseteq D_N \stackrel{\rm def}{=} D^\circ_N \cup \Delta_N$ for each $N$, and 
\begin{equation}\label{1.40a}
\sup\limits_{N \ge 1} \; \sup\limits_{x \in \Delta_N} g(x,x) \le g_0 < \infty.
\end{equation}

\n
The proof of Lemma \ref{lem1.4} can be easily adapted to show that (\ref{1.34}) remains true if one replaces (\ref{1.33}) by (\ref{1.40a}). Indeed, one simply modifies the above definition of the stopping times $S_\ell$, $\ell \ge 0$, in (\ref{1.39}) so that setting $\Delta = \bigcup_{N \ge 1} \Delta_N$, $S_0 = \inf\{k \ge 0$; $Z_k \in \Delta\}$ and for $\ell \ge 0$, $S_{\ell + 1} = \inf\{k > S_\ell$; $Z_k \in \Delta$ and $\sum_{1 \le j \le \ell} g(Z_k, Z_{S_j}) \le \frac{g_0}{2}\}$. Noting that the walk $Z$ visits $\Delta$ infinitely often $P_x$-a.s., for any $x \in E$, the proof proceeds as below (\ref{1.39}). 

\medskip
Let us mention that the condition (\ref{1.40a}) is also equivalent to 
\begin{equation}\label{1.40b}
\{x \in E; g(x,x) > g_0\} \;\mbox{has no unbounded component}
\end{equation}

\n
(clearly (\ref{1.40a}) implies (\ref{1.40b}) and conversely, one defines by induction finite connected sets $C_N$ in $E$ containing $x_0$, with outer boundary $\Delta_N$, via $C_1$ consists of $x_0$ and the points linked to $x_0$ by a path where $g(\cdot,\cdot) > g_0$ prior to reaching $x_0$, and $C_{N+1}$ is the union of $C_N$, $\Delta_N$, and the points linked to $\Delta_N$ by a path where $g(\cdot,\cdot) > g_0$ prior to reaching $\Delta_N$. Then $D^\circ_N$ is defined as the connected component of $x_0$ in $\wt{E} \backslash \Delta_N$, and (\ref{1.40a}) holds).

\hfill $\square$
\end{remark}

\medskip
The above Lemma \ref{lem1.4} will be very helpful in the next section, see for instance Lemma \ref{lem1.5}, and also below (\ref{2.39}) in Theorem \ref{theo2.3}.

\section{The coupling}
\setcounter{equation}{0}

In this section we construct a coupling between the Gaussian free field on $\wt{E}$ and random interlacements at level $u$ on $\wt{E}$, which refines (\ref{1.26}) - (\ref{1.29}). The main result appears in Theorem \ref{theo2.3} and in essence follows from the application of the strong Markov property of the Gaussian free field. As a direct consequence of this main result we obtain a coupling of the Gaussian free field and random interlacements at level $u$ on $E$, which fulfills (\ref{0.8}), see Corollary \ref{cor2.4}, and thus refines (0.7) which was proven in \cite{Lupu}. This corollary will play an important role for the derivation of the upper bound (\ref{4.16}) in Theorem \ref{theo4.3}, when $E$ is a $(d+1)$-regular tree with unit weights.

\medskip
We keep the notation from the previous section and tacitly assume that (\ref{1.31}) holds (i.e. $\wt{\IP}^G$-a.s., $\{\wt{\varphi} > 0\}$ only has bounded components), and (\ref{1.33}) holds (i.e. $g(x,x)$ is uniformly bounded for $x \in E$, see however Remark \ref{rem2.2} 2) and Remark \ref{rem2.6} for a weakening of this assumption). We consider a given $u > 0$, cf.~(\ref{1.20}), and define for $\o \in \wt{\Omega}$
\begin{equation}\label{1.41}
\begin{split}
\cC_\infty(\o) & = \mbox{the closure of the union of unbounded components of $\{\wt{\varphi} < \sqrt{2u}\}$}
\\
& \,\quad \mbox{(a closed subset of $\wt{E}$)},
\\[1ex]
\cU_\infty (\o) & = \wt{E} \backslash \cC_\infty(\o) \; \mbox{(an open subset of $\wt{E}$)}.
\end{split}
\end{equation}
Since $\wt{\varphi}$ is continuous, for all $\o \in \wt{\Omega}$, 
\begin{equation}\label{2.1}
\wt{\varphi}_z = \sqrt{2u} \;\;\mbox{for $z \in \partial \cC_\infty(\o)$ (the boundary of $\cC_\infty(\o)$)}.
\end{equation}

\n
The next proposition contains the first step in the proof of Theorem \ref{theo2.3}, and mainly relies on a suitable application of the strong Markov property to certain compatible compact subsets of $\wt{E}$ approximating $\cC_\infty$. We refer to Remark \ref{rem2.2} for further comments about the interpretation of Proposition \ref{prop2.1}.

\begin{proposition}\label{prop2.1}
For $z,z' \in \wt{E}$ and $\o \in \wt{\Omega}$, define
\begin{align}
h_\infty(z,\o) & = \sqrt{2u} \;1\{z \in \cU_\infty(\o)\} + \wt{\varphi}_z \;1\{z \in \cC_\infty(\o)\}, \label{2.2}
\\[1ex]
g_\infty(z,z',\o) & = \wt{g}_{\cU_\infty(\o)}(z,z') \;\mbox{(with the notation above (\ref{1.2}))}. \label{2.3}
\end{align}

\n
Then, $h_\infty(z,\cdot)$ and $g_\infty(z,z', \cdot)$ are measurable functions on $\wt{\Omega}$ and for all $z_1,\dots,z_M \in \wt{E}$ and $a_1,\dots,a_M \in \IR$, one has
\begin{equation}\label{2.4}
\wt{\IE}^G [e^{i\,\sum\limits^M_{j=1} a_j \wt{\varphi}_{z_j}}] = \wt{\IE}^G [e^{i \,\sum\limits^M_{j=1} a_j h_\infty(z_j) - \frac{1}{2} \,\sum\limits^M_{j,\ell = 1} a_j a_\ell g_\infty(z_j, z_\ell)}]
\end{equation}

\n
(of course the left-hand side of (\ref{2.4}) coincides with $e^{-\frac{1}{2} \,\sum^M_{j,\ell = 1} a_j a_\ell \wt{g}(z_j,z_\ell)}$).
\end{proposition}

\begin{proof}
We consider some base point $x_0 \in E$, and in the notation of (\ref{1.1}) introduce
\begin{equation}\label{2.5}
B_N = B_N (x_0), \;B^\circ_N = B^\circ_N(x_0), \; S_N = S_N(x_0).
\end{equation}

\n
For $1 \le N < L$ integers, we define for $\o \in \wt{\Omega}$, 
\begin{align}
\cC_{N,L} = (B_L \backslash B^\circ_N) \cup \mbox{\small $\bigcup\limits_{x \in S_N}$} & \; \mbox{(the closure of the connected component} \label{2.6}
\\[-2ex]
&\;\;\;\mbox{of $\wt{\varphi} < \sqrt{2 u}$ in $B_N$ that contains $x$)}, \nonumber
\end{align}

\n
where the last set in parenthesis is understood as empty, when $\wt{\varphi}_x \ge \sqrt{2u}$,
\begin{align}
\cC_N =   \mbox{\small $\bigcup\limits_{L > N}$} \;\cC_{N,L} = (\wt{E} \backslash B^\circ_N) \cup  \mbox{\small $\bigcup\limits_{x \in S_N}$} \;(& \mbox{the closure of the connected component} \label{2.7}
\\[-2ex]
&\mbox{of $\wt{\varphi} < \sqrt{2 u}$ in $B_N$ that contains $x$)}. \nonumber
\end{align}
We also define
\begin{equation}\label{2.8}
\cU_{N,L} = \wt{E} \backslash \cC_{N,L} \; \mbox{and} \; \cU_N = \wt{E} \backslash \cC_N.
\end{equation}

\n
Note that $\cC_{N,L}$ is a compact subset of $\wt{E}$, with a finite number of connected components, that it is the closure of its interior, and that (see above (\ref{1.10}))
\begin{equation}\label{2.9}
\mbox{$\cC_{N,L}$ is compatible.}
\end{equation}

\n
Indeed, as we briefly explain $\{\cC_{N,L} \subseteq O\} \in \cA_O$, when $O \subseteq \wt{E}$ is open. We can assume that $O \supseteq B_L \backslash B^\circ_N$ (otherwise $\{\cC_{N,L} \subseteq O\}$ is empty), and consider (with hopefully obvious notation) the finitely many ``polygonal paths'' that start in $S_N$ and then remain in $B^\circ_N \cap O$ except for their final point that belongs to $\partial O \cap B^\circ_N$. Then, $\{\cC_{N,L} \subseteq O\}$ is realized exactly when for all such paths $\wt{\varphi}_z \ge \sqrt{2u}$ for some $z$ on the path strictly before the final point of the path. Hence, $\{\cC_{N,L} \subseteq O\}$ is $\sigma(\wt{\varphi}_z,z \in O) = \cA_O$-measurable, and (\ref{2.9}) follows.

\medskip
We further note that while $\cC_{N,L}$ increases to $\cC_N$, when $L > N$ goes to $\infty$, 
\begin{equation}\label{2.10}
\mbox{$\cC_N$ is a decreasing sequence and $\mbox{\small $\bigcap\limits_{N \ge 1}$} \,\cC_N = \cC_\infty$}.
\end{equation}

\n
Indeed, by direct inspection, $\cC_N$ is a decreasing sequence that contains $\cC_\infty$. Conversely, when $z$ belongs to $\bigcap_N \cC_N$, then for any $N$ one can find a ``polygonal path'' from $z$ to $S_N$ on which $\wt{\varphi} < \sqrt{2u}$, except maybe at its starting point where $\wt{\varphi}_z \le \sqrt{2u}$. By the drawer (or pigeonhole) principle one thus finds that $z \in \cC_\infty$, and (\ref{2.10}) follows.

\medskip
We now turn to the measurability statement below (\ref{2.3}). By (\ref{1.12}) ii) applied to $\cC_{N,L}$, together with (\ref{2.10}) and the fact that $\cC_{N,L}$ increases to $\cC_N$ as $L \r \infty$, we see that for any $z \in \wt{E}$, $h_\infty(z,\cdot)$ is $\cA$-measurable. Moreover, since $\cC_{N,L}$ is compatible, by (\ref{1.17}), for any $z,z' \in \wt{E}$, $\wt{g}_{\cU_{N,L}}(z,z')$ is $\cA$-measurable. Further, when $z, z' \in B^\circ_N$, then the connected components of $z$ (resp.~of $z'$) in $\cU_{N,L}$ and $\cU_N$ coincide so that
\begin{equation}\label{2.11}
\mbox{$\wt{g}_{\cU_{N,L}} (z,z') = \wt{g}_{\cU_N}(z,z')$, when $z,z' \in B^\circ_N$, $1 \le N < L$}.
\end{equation}

\n
By (\ref{2.10}) and the monotone convergence (\ref{1.4}), we see that for $z,z' \in B^\circ_N$
\begin{equation}\label{2.12}
g_\infty(z,z') = \wt{g}_{\cU_\infty}(z,z') = \lim\limits_N \uparrow \wt{g}_{\cU_N}(z,z'),
\end{equation}
and in particular $g_\infty (z,z',\cdot)$ is $\cA$-measurable.

\medskip
We will now prove (\ref{2.4}). We assume $N$ so large that $z_1,\dots,z_M \in B^\circ_N$. The observation above (\ref{2.11}) now yields that for any $z \in B^\circ_N$,
\begin{equation}\label{2.13}
h_{\cU_{N,L}}(z) = \wt{E}_z[\wt{\varphi}(X_{T_{\cU_{N,L}}})] = \wt{E}_z[\wt{\varphi}(X_{T_{\cU_N}})] = h_{\cU_N}(z).
\end{equation}
The strong Markov property (\ref{1.18}) applied to $\cC_{N,L}$ then implies that
\begin{equation}\label{2.14}
\begin{split}
\wt{\IE}^G[e^{i \sum\limits^M_{j=1} a_j \wt{\varphi}_{z_j}}] & = \wt{\IE}^G[e^{i \sum\limits^M_{j=1} a_j  h_{\cU_{N,L}(z_j)} - \frac{1}{2} \, \sum\limits^M_{j,\ell = 1} a_j a_\ell \,\wt{g}_{\cU_{N,L}}(z_j,z_\ell)}]
\\
& = \wt{\IE}^G[e^{i \sum\limits^M_{j=1} a_j  h_{\cU_{N}(z_j)} - \frac{1}{2} \, \sum\limits^M_{j,\ell} a_j a_\ell \,\wt{g}_{\cU_{N}}(z_j,z_\ell)}],
\end{split}
\end{equation}
using (\ref{2.11}), (\ref{2.13}) in the last step.

\medskip
We already know that $\wt{g}_{\cU_N}(z_j,z_\ell)$ increases to $g_\infty (z_j,z_\ell)$, for $N \r \infty$. We also show
\begin{equation}\label{2.15}
\mbox{$\lim\limits_N h_{\cU_N} (z_j) = h_\infty(z_j)$ in $\wt{\IP}^G$-probability, for $1 \le j \le M$}.
\end{equation}

\n
Letting $N$ tend to infinity in the last line of (\ref{2.14}) will yield (\ref{2.4}), and conclude the proof of Proposition \ref{prop2.1}. We now prove (\ref{2.15}).

\medskip
We first observe that $T_{\cU_N}$ is a jointly measurable function of $\o$ and $X$. Indeed, it vanishes if $X_0 \notin B^\circ_N$, and if $X_0 \in B^\circ_N$, then $T_{\cU_N} = T_{\cU_{N,L}}$ is the non-decreasing limit of the $T_{\cU^n_{N,L}}$, where $\cU^n_{N,L} = \wt{E} \backslash \cC^n_{N,L}$ in the notation of (\ref{1.14}), and the $T_{\cU^n_{N,L}}$ are jointly measurable. With this observation, we will then be able to apply Fubini's theorem in the calculations below. The same observation applies to $T_{\cU_\infty} = \lim_n \uparrow T_{\cU_N}$, cf.~(\ref{2.10}).

\medskip
For $z \in B^\circ_N$, we write
\begin{equation}\label{2.16}
\begin{split}
h_{\cU_N}(z) =   \mbox{$A_1 + A_2$, where $A_1$}  =& \;\wt{E}_z[\wt{\varphi}(X_{T_{\cU_N}}), \;T_{\cU_N} < T_{B^\circ_N}] \;\mbox{and}
\\
 A_2 = &\; \wt{E}_z [\wt{\varphi}(X_{T_{B^\circ_N}}), \; T_{\cU_N} = T_{B^\circ_N}].
\end{split}
\end{equation}

\n
We first consider $A_1$. Note that when $z \in B^\circ_N \backslash \cC_N$, then $\wt{P}_z$-a.s. on $\{T_{\cU_N} < T_{B^\circ_N}\}, \;X_{T_{\cU_N}} \in \partial \cC_N \cap B^\circ_N$ and hence $\varphi(X_{T_{\cU_N}}) = \sqrt{2u}$, see (\ref{2.7}). As a result, for $z \in B^\circ_N$
\begin{equation}\label{2.17}
A_1 = \wt{\varphi}_z \,1\{z \in \cC_N\} + \sqrt{2u} \;1 \{z \notin \cC_N\} \;\wt{P}_z [H_{\cC_N} < T_{B^\circ_N}].
\end{equation}

We also have
\begin{lemma}\label{lem1.5}
\begin{equation}\label{1.42}
\mbox{$\wt{\IP}^G$-a.s., for all $z \in \wt{E}$, $\wt{P}_z [H_{\cC_\infty} < \infty] = 1$}.
\end{equation}
\end{lemma}

\begin{proof} Note that $H_{\cC_\infty} = T_{\cU_\infty}$ is jointly measurable in $\o$ and $X$, see above (\ref{2.16}), and the probability in (\ref{1.42}) is a measurable function of $\o$, which is continuous in $z$ (all points of $\cC_\infty$ are regular for $X$). By (\ref{1.32}) we have $\wt{\cI}^u \subseteq \{\wt{\gamma} < \sqrt{2u}\}$ where $\wt{\gamma}$ is distributed as $\wt{\varphi}$ under $\wt{\IP}^G$, and $\wt{\cI}^u = \{x \in \wt{E}; \wt{\ell}_{z,u} > 0\}$ only has unbounded components and $\wt{\ell}$ has same law as under $\wt{\IP}^I$. The claim (\ref{1.42}) now follows from (\ref{1.34}). 
\end{proof}

By (\ref{1.42}) we know that on a set of full $\wt{\IP}^G$-measure
\begin{equation}\label{2.18}
\wt{P}_z [T_{\cU_N} < T_{B^\circ_N}] \stackrel{(\ref{2.10})}{\ge} \wt{P}_z [T_{\cU_\infty} < T_{B^\circ_N}] = \wt{P}_z[H_{\cC_\infty} < T_{B^\circ_N}] \underset{N}{\stackrel{(\ref{1.42})}{\longrightarrow}} 1.
\end{equation}
Thus, by (\ref{2.10}) and (\ref{2.17}), we find that
\begin{equation}\label{2.19}
\mbox{$\wt{\IP}^G$-a.s.,} \;A_1 \underset{N}{\longrightarrow} \wt{\varphi}_z \,1 \{z \in \cC_\infty\} + \sqrt{2u} \; 1\{z \notin \cC_\infty\} = h_\infty(z).
\end{equation}
We now turn to $A_2$. By Cauchy-Schwarz's inequality we have
\begin{equation}\label{2.20}
\begin{split}
\wt{\IE}^G[|A_2|] & \le \wt{\IE}^G \wt{E}_z [\wtv^2(X_{T_{B^\circ_N}})]^{\frac{1}{2}} \,\wt{\IP}^G \! \otimes \wt{P}_z [T_{\cU_N} = T_{B^\circ_N}]^{\frac{1}{2}}
\\
&\!\!\! \stackrel{(\ref{1.33})}{\le} \sqrt{g_0} \,\wt{\IP}^G \! \otimes \wt{P}_z [T_{\cU_N} = T_{B^\circ_N}]^{\frac{1}{2}} \underset{N}{\stackrel{(\ref{2.18})}{\longrightarrow}} 0.
\end{split}
\end{equation}

\n
Combining (\ref{2.19}), (\ref{2.20}), we have proved (\ref{2.15}) and Proposition \ref{prop2.1} follows.
\end{proof}

\begin{remark}\label{rem2.2}  \rm ~

\medskip\n
1) As we now explain, a certain strong Markov property, see (\ref{2.21}) below, underlies Proposition \ref{prop2.1}. Observe that $\cC_{N,L}$ is non-decreasing in $L$, see (\ref{2.7}), and hence, by (\ref{1.13}), the $\cA^+_{\cC_{N,L}}$ are non-decreasing in $L > N$. We set $\cH_N = \sigma(\bigcup_{L > N} \cA^+_{\cC_{N,L}})$. Then, for $z \in B^\circ_N$, it follows from (\ref{2.13}) and the martingale convergence theorem that
\begin{equation*}
h_{\cU_N}(z) = h_{\cU_{N,L}}(z) \stackrel{(\ref{1.18})}{=} \wtie^G [\wtv_z | \cA^+_{\cC_{N,L}}] \underset{L}{\longrightarrow} \wtie^G[\wtv_z | \cH_N],
\end{equation*}

\n
and the equality $h_{\cU_N}(z) = \wtie^G [\wtv_z | \cH_N]$ extends to all $z \in \wte$, since both members coincide with $\wtv_z$, when $z \notin B^\circ_N$ (see (\ref{1.12}) iii).

\medskip
On the other hand, for fixed $L$, $\cC_{N,L}$ is non-increasing in $N (< L)$, see (\ref{2.6}), and therefore, by (\ref{1.13}), the $\sigma$-algebras $\cA^+_{\cC_{N,L}}$ decrease in $N( < L)$. It then follows that $\cH_N$ decreases with $N$ so that by the martingale convergence theorem for reverse martingales, for any $z \in \wte$, $h_{\cU_N}(z) \, \underset{\rm a.s.}{\stackrel{L^2(\wt{\IP}^G)}{\longrightarrow}} \,\wtie^G[\wtv_z | \cH_\infty]$, where $\cH_\infty = \bigcap_N \cH_N$. We have used the assumptions (\ref{1.31}) and  (\ref{1.33}) in the proof of (\ref{2.15}) to identify $\wtie^G[\wtv_z | \cH_\infty]$ with $h_\infty(z)$ from (\ref{2.2}). Moreover, by (\ref{2.12}) and the argument above (\ref{2.12}), one sees that $g_\infty(z,z')$ is $\cH_\infty$-measurable. In this light the proof of Proposition \ref{prop2.1} can be modified to show that (compare with (\ref{1.18}))
\begin{align}
&\mbox{under $\wt{\IP}^G$, conditionally on $\cH_\infty$, $(\wtv_z)_{z \in \wte}$ is a Gaussian field with}\label{2.21}
\\
&\mbox{mean $(h_\infty(z))_{z \in \wte}$ and covariance $g_\infty(\cdot,\cdot)$.} \nonumber
\end{align}

\medskip\n
We will actually not need (\ref{2.21}) in what follows, but this statement provides an additional interpretation to what Proposition \ref{prop2.1} does. 

\bigskip\n
2) Note that Proposition \ref{prop2.1} holds as well if one replaces (\ref{1.33}) with the weaker  condition (\ref{1.40a}) (or equivalently (\ref{1.40b})). In essence, one simply replaces in the proof $B^\circ_N$ by $D^\circ_N$, $B_N$ by $D_N$ and $S_N$ by $\Delta_N$. One knows that Lemma \ref{lem1.4} holds (see Remark \ref{rem1.5}), and (\ref{2.20}) works as well in this new set-up. \hfill $\square$
\end{remark}

We will now construct the coupling announced at the beginning of this section. We consider the product space $\wt{\Omega} \times \wt{W}$ endowed with the product $\sigma$-algebra $\cA \otimes \cB$, and the product-measure $\wt{\IP}^G \otimes \wt{\IP}^I$. On this space we have the independent Gaussian free field $(\wtv_z)_{z \in \wte}$ and the field of local times $(\wt{\ell}_{z,u})_{z \in \wte}$ of random interlacements at level $u$, such that (\ref{1.21}), (\ref{1.22}) hold. By (\ref{1.21}) and the local finiteness of $E$,
\begin{equation}\label{2.22}
\mbox{$\partial \wt{\cI}^u$ is a locally finite subset of $\wte$ (see (\ref{1.21}) for notation).}
\end{equation}
We introduce the open subset of $\wte$
\begin{equation}\label{2.23}
\begin{split}
\cJ &\!\! = \mbox{the union of $\wt{\cI}^u$ and all connected components of $\{|\wtv| > 0\}$ intersecting $\partial \wt{\cI}^u$}
\\
& \!\! = \mbox{the union of the connected components of $\{2 \wt{\ell}_{\point,u} + \wtv^2_\point > 0\}$ intersecting $\wt{\cI}^u$}.
\end{split}
\end{equation}
By a similar argument as in (\ref{2.22})
\begin{equation}\label{2.24}
\mbox{$\partial \cJ$ is a locally finite subset of $\wte$},
\end{equation}
and we define the closed set
\begin{equation}\label{2.25}
\cC'_\infty = \cJ \cup \partial \cJ.
\end{equation}
Note that by the second line of (\ref{2.23}),
\begin{equation}\label{2.26}
\mbox{$\wt{\ell}_{z,u} = 0$ and $\wtv_z = 0$, for all $z \in \partial \cJ$}.
\end{equation}
We then define the random field
\begin{equation}\label{2.27}
\wt{\eta}_z = \big(\sqrt{2u} - \mbox{\f $\sqrt{2 \wt{\ell}_{z,u} + \wtv^2_z}$}\big) \,1\{z \in \cC'_\infty\} + (\wtv_z + \sqrt{2u}) \, 1\{z \notin \cC'_\infty\}
\end{equation}

\medskip\n
and note that by (\ref{2.26}) the expression above remains unchanged if $\cC'_\infty$ is replaced by $\cJ$.

\medskip
The main result of this section, i.e. Theorem \ref{theo2.3} below, states that $\wt{\eta}$ is a Gaussian free field on $\wte$. Unlike $\wt{\gamma}$ in (\ref{1.26}) - (\ref{1.29}), $\wt{\eta}$ is solely defined in terms of $\wt{\ell}_{\point, u}$ and $\wtv$. Actually, by (\ref{1.29}) and (\ref{1.32}), $\wt{\eta}$ amounts to a suitable resampling of the signs of $\wt{\gamma} - \sqrt{2u}$ outside $\cC'_\infty = \cJ \cup \partial \cJ$, which a.s. coincides with the closure of the union of the unbounded components of $\{\wt{\gamma} < \sqrt{2u}\}$, see (\ref{2.45}), (\ref{2.46}) below.  The proof of Theorem \ref{theo2.3} will mainly rely on the application of the strong Markov property of $\wtv$ and Proposition \ref{prop2.1}. We recall (\ref{1.3}) for notation.

\begin{theorem}\label{theo2.3}
\begin{align}
& \mbox{$z \rightarrow \wt{\eta}_z$ is a continuous function on $\wte$ (which equals $\sqrt{2u}$ on $\partial \cJ$)}. \label{2.28}
\\[1ex]
& \mbox{The law of $\wt{\eta}$ on the canonical space $\wt{\Omega}$ is $\wt{\IP}^G$}. \label{2.29}
\end{align}
\end{theorem}

\begin{proof}
Let us prove (\ref{2.28}). We only need to check the continuity of $\wt{\eta}$ at all $z \in \partial \cJ$, which is locally finite (see (\ref{2.24})), and that $\wt{\eta} = \sqrt{2u}$. But the expressions between parenthesis in the right-hand side of (\ref{2.27}) are continuous and take value $\sqrt{2u}$ on $\partial \cJ$.  The claim (\ref{2.28}) follows. We then turn to (\ref{2.29}). It suffices to show that for $z_1,\dots,z_M \in \wte$
\begin{equation}\label{2.30}
\begin{array}{l}
\mbox{$(\wt{\eta}_{z_1},\dots,\wt{\eta}_{z_M})$ is a centered Gaussian vector with covariance} 
\\[0.5ex]
\mbox{matrix $\wt{g}(z_i,z_j), 1 \le i, j \le M$}.
\end{array}
\end{equation}

\n
For this purpose we will derive a formula for $\wt{\eta}$ similar to (\ref{2.4}) with $\cC'_\infty$ and $\cU'_\infty = \wte \backslash \cC'_\infty$ in place of $\cC_\infty$ and $\cU_\infty$, cf.~(\ref{2.43}). We first approximate $\cC'_\infty$. For $1 \le N < L$ integers, we define (see (\ref{2.5}) for notation)
\begin{align}
\cC'_{N,L} = &\; (B_L \backslash B^\circ_N) \cup \big(B_L \cap (\wt{\cI}^u \cup \partial \wt{\cI}^u)\big) \cup \mbox{\small $\bigcup_{z \in \partial \wt{\cI}^u \cap B^\circ_L}$}  \label{2.31}
\\
&\; \mbox{(the closure of the connected component of $\{| \wtv | > 0\} \cap B^\circ_L$ containing $z$)}\nonumber
\intertext{and the last set in parenthesis is understood as empty if $\wtv_z = 0$, as well as}
\cC'_{N} = &\;(\wte \backslash B^\circ_N) \cup (\wt{\cI}^u \cap \partial \wt{\cI}^u)   \cup \mbox{\small $\bigcup_{z \in \partial \wt{\cI}^u}$} \label{2.32}
\\
&\; \mbox{(the closure of the connected component of $\{| \wtv | > 0\}$ containing $z$)}. \nonumber
\end{align}

\n
Observe that the $\cC'_{N,L}$ are compact and non-decreasing in $L$. In fact, in (\ref{2.31}) one can replace in the last union (over $z \in \partial \wt{\cI}^u \cap B^\circ_L$) all the terms where the set in parenthesis is not contained in $B_N$ by a union of the closures of the connected components of $\{|\wtv| > 0\} \cap B^\circ_L$ containing $x$, for all $x \in S_N$ such that the connected component of $\{|\wtv | > 0\}  \cap B^\circ_L$ containing $x$ meets $\partial \wt{\cI}^u \cap B^\circ_L$. One can also make a similar replacement for (\ref{2.32}) and replace in the last union all the terms where the set in parenthesis is not contained in $B_N$ by a union of the closures of the connected components of $\{|\wtv |> 0\}$ containing $x \in S_N$ such that the connected component of $\{|\wtv | > 0\}$ containing $x$ meets $\partial \wt{\cI}^u$. From this observation one sees that	
\begin{align}
&\mbox{$\cC'_{N,L} \cap B_N = \cC'_N \cap B_N$, for large $L$, and} \label{2.33}
\\[1ex]
& \mbox{$\cC'_{N,L}$ increases with $L > N$ and \mbox{\small $\bigcup_{L > N}$}  \,$\cC'_{N,L} = \cC'_N$}. \label{2.34}
\end{align}																														\n\n
Moreover, the $\cC'_N$ are closed (for instance by (\ref{2.33})), decrease with $N$ and
\begin{equation}\label{2.35}
\mbox{\small $\bigcap\limits_{N \ge 1}$} \, \cC'_N = \cC'_\infty
\end{equation}														(simply note that $\cJ \subseteq \bigcap_N \cC'_N \subseteq \cJ \cup \partial \cJ = \cC'_\infty$).

\medskip
An important additional observation is that for $w \in \wt{W}$ (i.e. ``freezing'' the interlacement and in particular $\wt{\cI}^u$), as a function of $\o \in \wt{\Omega}$,	\begin{equation}\label{2.36}
\mbox{$\cC'_{N,L}$ is a compatible compact subset of $\wte$}
\end{equation}		

\n
(note that when $w$ is fixed, $\partial \wt{\cI}^u \cap B^\circ_L$ is a finite deterministic set, and we can use a similar argument as below (\ref{2.9})). One also has the fact that for $z,z' \in \wt{E}$, $1\{z \in \cC'_\infty\}$ and $\wt{g}_{\cU'_\infty}(z,z')$ (with $\cU'_\infty$ as in (\ref{2.37})) are jointly measurable in  $\o$ and $w$.

\medskip
We will apply the strong Markov property under $\wt{\IP}^G$ and define the open sets
\begin{equation}\label{2.37}
\cU'_{N,L} = \wte \backslash \cC'_{N,L}, \;\cU'_N = \wte \backslash \cC'_N, \;\mbox{and} \; \cU'_\infty = \wte \backslash \cC'_\infty,
\end{equation}	
and for any $J \subseteq \{1,\dots,M\}$ the events
\begin{equation}\label{2.38}
\mbox{$A^J_{N,L}= \{z_\ell \in \cU'_{N,L}$, when $\ell \in J$, and $z_\ell \notin \cU'_{N,L}$, when $\ell \notin J\}$}
\end{equation}		

\n
as well as $A^J_N$ and $A^J_\infty$, with $\cU'_N$ and $\cU'_\infty$ respectively in place of $\cU'_{N,L}$ in (\ref{2.38}).

\medskip
Note that $A^J_{N,L} \in \cA^+_{\cC'_{N,L}}$ by (\ref{1.12}), and by the strong Markov property (\ref{1.18}), for all $J \subseteq \{1,\dots,M\}$ and $a_1,\dots,a_M \in \IR$ one has
\begin{equation}\label{2.39}
\begin{array}{l}
\wtie^G[e^{i (\sum\limits_{j\in J} a_j (\sqrt{2 u} + \wtv_{z_j}) + \sum\limits_{j \notin J} a_j(\sqrt{2 u} - \sqrt{2 \wt{\ell}_{z_j,u} + \wtv^2_{z_j}}\,))}, A^J_{N,L} ] = 
\\[2ex]
\wtie^G[e^{i (\sum\limits_{j\in J} a_j (\sqrt{2 u} + h_{\cU'_{N,L}} (z_j)) + \sum\limits_{j \notin J} a_j(\sqrt{2 u} - \sqrt{2 \wt{\ell}_{z_j,u} + \wtv^2_{z_j}}\,)) - \frac{1}{2} \,\sum\limits_{j,\ell \in J} a_j a_\ell \wt{g}_{\cU'_{N,L}}(z_j,z_\ell) }, A^J_{N,L} ] .
\end{array}
\end{equation}

\n
If $N$ is so large that $z_1,\dots, z_M \in B^\circ_N$, then by (\ref{2.33}), letting $L$ tend to $\infty$, we can replace in (\ref{2.39}) $A^J_{N,L}$ with $A^J_N$, $h_{\cU'_{N,L}}(z_j)$ with $h_{\cU'_N} (z_j)$, and $\wt{g}_{\cU'_{N,L}}(z_j,z_\ell)$ with $\wt{g}_{\cU'_N} (z_j,z_\ell)$. Moreover, with similar arguments as for the proof of (\ref{2.15}), with now (\ref{1.34}) in place of (\ref{1.42}), we find that for $\wt{\IP}^I$-a.e. $w \in \wt{W}$, and $1 \le j, \ell \le M$
\begin{align}
&h_{\cU'_N}(z_j) \stackrel{\wt{\IP}^G \,{\rm-prob.}}{\longrightarrow} 0 \cdot 1\{z_j \in \cU'_\infty\}+ \wtv_{z_j} \,1\{z_j \in \cC'_\infty\} \label{2.40}
\\[2ex]
&\wt{g}_{\cU'_N}(z_j,z_\ell) \underset{N}{\longrightarrow} \wt{g}_{\cU'_\infty}(z_j,z_\ell) \quad \mbox{(by (\ref{1.4}) and (\ref{2.35})).} \label{2.41}
\end{align}

\n
Letting $N$ tend to infinity, we find with (\ref{2.35}) that for $J \subseteq \{1,\dots, M\}$ and $\wt{\IP}^I$-a.e. $w$,
\begin{equation}\label{2.42}
\begin{array}{l}
\wtie^G[e^{i (\sum\limits_{j\in J} a_j (\sqrt{2 u} + \wtv_{z_j}) + \sum\limits_{j \notin J} a_j(\sqrt{2 u} - \sqrt{2 \wt{\ell}_{z_j,u} + \wtv^2_{z_j}}\,))}, A^J_\infty] = 
\\[2ex]
\wtie^G[e^{i (\sum\limits_{j\in J} a_j \sqrt{2 u} + \sum\limits_{j \notin J} (\sqrt{2 u} - \sqrt{2 \wt{\ell}_{z_j,u} + \wtv^2_{z_j}}\,)) - \frac{1}{2} \,\sum\limits_{j,\ell \in J}^M a_j a_\ell \wt{g}_{\cU'_\infty}(z_j,z_\ell) }, A^J_\infty] .
\end{array}
\end{equation}

\medskip\n
Summing over $J$, with the definition of $\wt{\eta}$ in (\ref{2.27}), we find that for $\wt{\IP}^I$-a.s. $w \in \wt{W}$,
\begin{equation}\label{2.43}
\wtie^G[e^{i \sum\limits^M_{j=1} a_j \wt{\eta}_{z_j}}] =   \wtie^G[e^{i \sum\limits^M_{j=1} a_j h'_\infty(z_j) - \frac{1}{2} \,\sum\limits^M_{j,\ell = 1} a_j a_\ell \,\wt{g}_{\cU'_\infty}(z_j, z_\ell)}],
\end{equation}
where for $z \in \wte$ we have set
\begin{equation}\label{2.44}
\begin{split}
h'_\infty(z) & = \sqrt{2u} \;1\{z \in \cU'_\infty\} + \wt{\eta}_z \;1\{z \in \cC'_\infty\}
\\
& = \sqrt{2u} \;1\{z \in \cU'_\infty\} +(\sqrt{2u} - \mbox{\f $\sqrt{2 \wt{\ell}_{z,u} + \wtv^2_z}$}\,) \;1\{z \in \cC'_\infty\}.
\end{split}
\end{equation}

\n
On the probability space $(\Sigma, \cF, \wt{\IQ})$, see above (\ref{1.26}), extension of the product space $\wt{\Omega} \times \wt{W}$, and endowed with the Gaussian free field $\wt{\gamma}$, cf.~(\ref{1.26}) - (\ref{1.29}), we set 
\begin{equation}\label{2.45}
\begin{split}
\cC^{\wt{\gamma}}_\infty = & \;\mbox{the closure of the union of unbounded connected components}
\\
&\; \mbox{of $\{\wt{\gamma} < \sqrt{2u}\}$.}
\end{split}
\end{equation}
As we now explain

\vspace{-3ex}
\begin{equation}\label{2.46}
\mbox{$\wt{\IQ}$-a.s., $\cC'_\infty = \cC^{\wt{\gamma}}_\infty$}.
\end{equation}

\medskip\n
First note that a.s. $\cC'_\infty \subseteq \cC^{\wt{\gamma}}_\infty$. Indeed, a.s. $\cJ \subseteq \cC_\infty^{\wt{\gamma}}$ because each component of $\cJ$ contains a component of $\wt{\cI}^u$, by (\ref{2.23}), which is unbounded by (\ref{1.21}), and hence by (\ref{1.29}), (\ref{1.32}), a.s., $\cC'_\infty = \cJ \cup \partial \cJ \subseteq \cC^{\wt{\gamma}}_\infty$. For the converse inclusion, note that any component of $\{\wt{\gamma} < \sqrt{2u}\}$ not intersecting $\cC'_\infty = \cJ \cup \partial \cJ$ lies in $\wt{E} \backslash \wt{\cI}^u$ and hence by (\ref{1.29}) is a.s.~contained in a connected component of $\{|\wtv | > 0\}$. By (\ref{1.31}) (applied to $\wtv$ and $- \wtv$), a.s.~all such components are bounded. Thus, a.s.~all unbounded components of $\{\wt{\gamma} < \sqrt{2u}\}$ intersect $\cC'_\infty$ and hence $\cJ$, so that by (\ref{1.29}), (\ref{2.23}), a.s. $\cC^{\wt{\gamma}}_\infty \subseteq \cC'_\infty$. The claim (\ref{2.46}) follows. 

\medskip
By (\ref{1.29}), (\ref{1.32}), we see that $\wt{\IQ}$-a.s, for $z \in \cC'_\infty$, $\wt{\gamma}_z = \sqrt{2u} - \mbox{\f $\sqrt{2 \wt{\ell}_{z,u} + \wtv^2_z}$}$, so with (\ref{2.44}) and (\ref{2.46}) we find that for $z,z' \in \wte$, $\wt{\IQ}$-a.s. (setting $\cU^{\wt{\gamma}}_\infty = \wte \backslash \cC^{\wt{\gamma}}_\infty$) 
\begin{equation*}
h'_\infty(z) = \sqrt{2u} \; 1\{z \in \cU^{\wt{\gamma}}_\infty\} + \wt{\gamma}_z \,1\{z \in \cC_\infty^{\wt{\gamma}}\} \;\mbox{and} \;  \wt{g}_{\cU'_\infty}(z,z') = \wt{g}_{\cU_\infty^{\wt{\gamma}}}(z,z').
\end{equation*}

\medskip\n
We can thus apply (\ref{2.4}) of Proposition \ref{prop2.1} to $\wt{\gamma}$, and after integration of (\ref{2.43}) with respect to $\wt{\IP}^I$ conclude that the Fourier transforms of $(\wt{\eta}_{z_1},\dots, \wt{\eta}_{z_M})$ and $(\wt{\gamma}_{z_1},\dots , \wt{\gamma}_{z_M})$ coincide. This shows (\ref{2.30}) and concludes the proof of Theorem \ref{theo2.3}.
\end{proof}

\medskip
By considering restrictions of $\wtv, \wt{\cI}^u$ (or $\wt{\ell}_{\point,u}$), $\wt{\eta}$ to $E$, we can for instance obtain the reinforcement (\ref{0.8}) of (\ref{0.7}). Namely, one has
\begin{corollary}\label{cor2.4}
One can couple independent copies $(\varphi_x)_{x \in E}$ and $\cI^u$ of the Gaussian free field on $E$ and random interlacements at level $u$ on $E$, with $(\eta_x)_{x \in E}$ a Gaussian free field on $E$, so that with $\cV^u = E \backslash \cI^u$ the vacant set at level $u$
\begin{equation}\label{2.47}
\mbox{for all $A \subseteq (0,\infty)$, $\{x \in E; \eta_x \in \sqrt{2u} + A\} \subseteq \{ x \in E; \varphi_x \in A\} \cap \cV^u$}.
\end{equation}
\end{corollary}

\begin{proof}
We denote by $\varphi,\ell,\eta$ the restrictions to $E$ of $\wt{\varphi},\wt{\ell}_{\point, u}, \wt{\eta}$ (defined on the product space $\wt{\Omega} \times \wt{W}$ as in Theorem \ref{theo2.3}), and set $\cI^u = \wt{\cI}^u \cap E$. By (\ref{2.27}), $\wt{\eta}_z \le \sqrt{2u}$ for all $z \in \cC'_\infty \supseteq \wt{\cI}^u$, so that
\begin{equation}\label{2.48}
\eta_x \,1\{\eta_x > \sqrt{2u}\} = (\varphi_x + \sqrt{2u}) \,1\{\varphi_x > 0, x \in \cV^u \backslash\cC'_\infty\}, \;\mbox{for all $x \in E$}.
\end{equation}

\smallskip\n
This readily implies (\ref{2.47}).
\end{proof}

\begin{remark}\label{rem2.6} \rm ~
Here again, with the observation made in Remark \ref{rem2.2} 2) both Theorem \ref{theo2.3} and Corollary \ref{cor2.4} remain true when one replaces assumption (\ref{1.33}) by the weaker but more technical assumption (\ref{1.40a}) (or equivalently by (\ref{1.40b})). \hfill $\square$
\end{remark}

\section{The regular tree case: some preparation}
\setcounter{equation}{0}

We will apply the result of the previous section to the study of level-set percolation of the Gaussian free field on a regular tree. This section contains some preparation. In particular, we introduce an important spectral quantity $\lambda_h$, study some of its properties in Proposition \ref{prop3.1}, and characterize the critical value $h_*$ for the level-set percolation of the Gaussian free field as the unique $h_*$ such that $\lambda_{h_*} = 1$ in Proposition \ref{prop3.3}. In essence, we investigate here a specific branching Markov chain on $\IR$ (involving an Ornstein-Uhlenbeck kernel) with a barrier.

\medskip
We keep similar notation as in the previous sections. We consider $d \ge 2$ and denote by $T$ the $(d+1)$-regular tree endowed with unit weights, so that $T$ plays the role of $E$. The canonical Gaussian free field $(\varphi_x)_{x \in T}$ is a centered Gaussian field on $T$ with covariance
\begin{equation}\label{3.1}
g(x,y) \stackrel{(\ref{0.2})}{=} \mbox{\f $\dis\frac{1}{(d+1)}$} \;E_x \Big[\dsl^\infty_{k=0} 1\{Z_k = y\}\Big], \;\mbox{for $x,y \in T$,}
\end{equation}
with $(Z_k)_{k \ge 0}$ the canonical walk on $T$, which starts at $x \in T$ under $P_x$.

\medskip
For $x \sim y$ in $T$, letting $H_y$ stand for the entrance time of $Z$ in $y$, one has
\begin{equation}\label{3.2}
P_x [H_y < \infty] = \mbox{\f $\dis\frac{1}{d+1}$} \; \Big(\mbox{\f $\dis\frac{d}{d+1}$}\Big)^{-1} = \mbox{\f $\dis\frac{1}{d}$} \,,
\end{equation}
and hence for all $x \in T$
\begin{equation}\label{3.3}
g(x,x) \stackrel{(\ref{3.1})}{=}  \mbox{\f $\dis\frac{1}{(d+1)}$} \;P_x [Z_k \not= x \;\mbox{for all} \; k \ge 1]^{-1} =  \mbox{\f $\dis\frac{1}{d+1}$} \,\Big( 1 - \mbox{\f $\dis\frac{1}{d}$}\Big)^{-1} = \mbox{\f $\dis\frac{d}{d^2-1}$} \;\stackrel{\rm def}{=} \sigma^2\,.
\end{equation}

\n
Given $x \sim x'$ in $T$, we write
\begin{equation}\label{3.4}
\begin{split}
T^+_{x,x'} & = \mbox{the set consisting of $x$ and its ``forward descendants''}
\\
& = \{y \in T; \;x' \notin [x,y]\},
\end{split}
\end{equation}

\n
where $[x,y]$ stands for the unique finite geodesic path on $T$ between $x$ and $y$. We also set
\begin{equation}\label{3.5}
T^-_{x,x'} = T \backslash T^+_{x,x'} \,.
\end{equation}
By the Markov property of the field $(\varphi_x)_{x \in T}$, one knows that
\begin{equation}\label{3.6}
\begin{array}{l}
\mbox{$(\varphi_y - P_y[H_{x'} < \infty] \varphi_{x'})_{y \in T^+_{x,x'}}$ is a centered Gaussian field independent}
\\
\mbox{of $\sigma(\varphi_{y'}, y' \in T^-_{x,x'})$, with covariance $g_{T^+_{x,x'}}(\cdot,\cdot)$},
\end{array}
\end{equation}

\n
where for $U \subseteq T$, $g_U(\cdot,\cdot)$ stands for the killed Green function outside $U$ (defined as in (\ref{3.1}), but now with a summation over $k \ge 0$ smaller than the exit time of $Z$ from $U$). In particular, choosing $y = x$, we have
\begin{align}
\varphi_x = \mbox{\f $\dis\frac{1}{d}$} \; \varphi_{x'} + \xi_{x,x'}, & \;\;\mbox{where $\xi_{x,x'}$ is independent of $\sigma(\varphi_{y'},y' \in T^-_{x,x'})$} \label{3.7}
\\[-1ex]
&\;\; \mbox{and has variance $\big(1 - \mbox{\f $\dis\frac{1}{d^2}$}\big) \sigma^2 \stackrel{(\ref{3.3})}{=} \mbox{\f $\dis\frac{1}{d}$}$}\,. \nonumber
\end{align}

\n
We then fix a base point $x_0 \in T$ and $x_{-1} \sim x_0$, and simply write
\begin{equation}\label{3.8}
T^+ = T^+_{x_0,x_{-1}}, \;T^- = T^-_{x_0,x_{-1}}, \; \xi = \xi_{x_0,x_{-1}} .
\end{equation}
For $n \ge 0$, we also have the set of $n$-th generation descendants of $x_0$ in $T^+$
\begin{equation}\label{3.9}
T^+_n = \{y \in T^+; \;d (x_0,y) = n\}.
\end{equation}

\medskip\n
We denote by $\nu$ the centered Gaussian law on $\IR$ with variance $\sigma^2$, and by $Q_t$, $t \ge 0$, the Ornstein-Uhlenbeck semigroup with variance $\sigma^2$, so that for $g$ bounded measurable on $\IR$
\begin{equation}\label{3.10}
Q_t \,g(a) = E^Y[g(a e^{-t} + \sqrt{1 - e^{-2t}} \,Y)], \;\mbox{for $a \in \IR, t \ge 0$},
\end{equation}

\n
where $Y$ is $\nu$-distributed and $E^Y$ stands for the expectation with respect to $Y$, see also \cite{IkedWata89}, p.~356. Actually, $Q_t$ extends as a self-adjoint contraction on $L^2(\nu)$ and admits the following expansion in an orthonormal basis of eigenfunctions (see \cite{IkedWata89}, p.~354-356):
\begin{equation}\label{3.11}
Q_t \,g(a) = \dsl_{n \ge 0} e^{-nt} h_{n,\sigma}(a) \big\langle h_{n,\sigma}, g\big\rangle_\nu, \;\mbox{for} \; a \in \IR, \, g \in L^2(\nu),
\end{equation}

\n
where $\langle \cdot,\cdot \rangle_\nu$ stands for the $L^2(\nu)$-scalar product and
\begin{equation}\label{3.12}
h_{n,\sigma}(a) = \sqrt{n!} \; H_n \,\big(\mbox{\f $\dis\frac{a}{\sigma}$}\big),
\end{equation}

\medskip\n
with $H_n(\cdot)$ the $n$-th Hermite polynomial (so that $\int_{\IR}  H_n (b) \, H_m (b) \,e^{-\frac{b^2}{2}} \,\frac{db}{\sqrt{2 \pi}} = \frac{1}{n!} \; \delta_{n,m}$, for $n,m \ge 0$, and $H_1(x) = x$, $H_2(x) = \frac{1}{2}\, (x^2-1)$, $H_3(x) = \frac{x^3}{6} - \frac{x}{2}$).

\medskip
In particular, (\ref{3.7}), (\ref{3.10}) yield that for all $x \sim x'$
\begin{align}
\mbox{for all $g \in L^2(\nu)$}, & \;\; \IE^G [ g(\varphi_x) \,| \,\sigma(\varphi_{y'}, y' \in T^-_{x,x'})] = (Q_{t_1} g) (\varphi_{x'}),\label{3.13}
\\
&\;\;\mbox{with} \; t_1 = \log \, \mbox{\f $d$} \quad \big(\mbox{so}\; e^{-t_1} = \mbox{\f $\dis\frac{1}{d}$}\big). \nonumber
\end{align}

\n
One also has hypercontradictivity estimates for the semigroup $(Q_t)_{t \ge 0}$, see \cite{IkedWata89}, p.~367:
\begin{equation}\label{3.14}
\| Q_t \, g\|_{L^q(\nu)} \le \|g\|_{L^p(\nu)}, \; \mbox{when} \; q-1 = (p-1) \,e^{2t} \;\mbox{and} \;1 < p < \infty.
\end{equation}

\n
We can then introduce the operators with $e^{-t_1} = 1/d$ as in (\ref{3.13}) and $h \in \IR$:
\begin{equation}\label{3.15}
\mbox{$L = d \,Q_{t_1}$, $\pi_h =$ multiplication by $1_{[h,\infty)}$, $L_h = \pi_h \, L \pi_h$}
\end{equation}

\n
(where $d\,Q_{t_1}$ denotes the multiplication of $Q_{t_1}$ by the scalar $d$).

\medskip
By (\ref{3.11}) we know that $L$ is a self-adjoint, non-negative Hilbert-Schmidt operator on $L^2(\nu)$, and by Theorem 6.22 or 6.23, p.~210 of \cite{ReedSimo79} that:
\begin{equation}\label{3.16}
\mbox{$L_h$ is a self-adjoint, non-negative, Hilbert-Schmidt operator on $L^2(\nu)$}.
\end{equation}

\n
We can then define the crucial quantity 
\begin{equation}\label{3.17}
\lambda_h = \|L_h\|_{L^2(\nu) \rightarrow L^2(\nu)} = \sup\big\{\big\langle g, L_h g 
\big\rangle_\nu; \;\|g\|_{L^2(\nu)} = 1\big\}, \; h \in \IR,
\end{equation}

\n
where $\| \cdot \|_{L^2(\nu) \r L^2(\nu)}$ denotes the operator norm and (\ref{3.17}) coincides with (\ref{0.9}) via the explicit calculation of $\big\langle g, L_h g\big\rangle_\nu$ with (\ref{3.10}).

\medskip
The next proposition will be very helpful.

\begin{proposition}\label{prop3.1}
For all $h \in \IR$, $\lambda_h$ is a simple eigenvalue of $L_h$ and there exists a unique $\chi_h \ge 0$, with unit $L^2(\nu)$-norm, continuous, positive on $[h,+ \infty)$, vanishing on $(-\infty,h)$, such that
\begin{equation}\label{3.18}
L_h \,\chi_h = \lambda_h\,\chi_h\,.
\end{equation}
In addition, for all $k \ge 0$, one has
\begin{align} 
&\|\chi_h\|_{L^{q_k}(\nu)} \le (d/\lambda_h)^k,\;\mbox{where} \; q_k = 1 + d^{2k}.\label{3.19}
\\[1ex]
&\mbox{The map $h \r \lambda_h$ is a decreasing homeomorphism of $\IR$ onto $(0,d)$}. \label{3.20}
\end{align}
\end{proposition}

\begin{proof}
We start with the first part of the proposition. We consider $h \in \IR$. By (\ref{3.16}) we can find $g$ with $\|g\|_{L^2(\nu)} = 1$, such that $L_h g = \lambda_h g$. If $g$ changes sign (i.e.~both $g^+$ and $g^-$ are not a.e.~$0$), then by (\ref{3.10}) (see also the expression in (\ref{0.9})), $\big\langle |g|, L_h |g|\big\rangle_\nu > \big\langle g,L_h g\big\rangle_\nu = \lambda_h$, with $|g|$ unit in $L^2(\nu)$, a contradiction. Hence, $g$ does not change sign and without loss of generality we can assume that it is non-negative.

\medskip
As we now explain, $\lambda_h$ is a simple eigenvalue. Indeed, if $f \in L^2(\nu)$ is an eigenfunction of $L_h$ attached to $\lambda_h$, we can choose $\alpha \in \IR$, so that $\big\langle f - \alpha g, 1\big\rangle_\nu = 0$. Then $f - \alpha g$ is an eigenfunction attached to $L_h$ and by the above paragraph, it does not change sign. It follows that $f- \alpha g = 0$ in $L^2(\nu)$ and $\lambda_h$ is a simple eigenvalue. Thus, $g$ is uniquely determined in $L^2(\nu)$ and we call it $\chi_h$. Note that 
\begin{equation}\label{3.21}
\begin{split}
\lambda_h \,\chi_h (a)  \stackrel{(\ref{3.10}),(\ref{3.15})}{=} &d\; \sqrt{\mbox{\f $\dis\frac{d}{2 \pi}$}} \;\dis\int^\infty_h e^{-\frac{d}{2} \,(b - \frac{a}{d})^2} \chi_h(b) \,db,
\\[1ex]
& \hspace{-6ex}= \mbox{\f $\dis\frac{d^2}{\sqrt{d^2-1}}$} \; \dis\int^\infty_h e^{-\frac{a^2}{2d} + ab - \frac{b^2}{2d}} \chi_h(b) \, d\nu(b), \ \mbox{when $a \ge h$, and}
\\[1.5ex]
&  \hspace{-6ex} = 0, \;\mbox{when $a < h$}.
\end{split}
\end{equation}

\n 
By the second and third line (and dominated convergence) we see that we can choose $\chi_h$ to be continuous and positive on $[h, + \infty)$ and equal to zero on $(-\infty,h)$. We now prove (\ref{3.19}). Note that $q_k = 1 + e^{2 t_1 k}$. We then use hypercontractivity, cf.~(\ref{3.14}), and the fact that $\pi_h$ contracts $L^q(\nu)$-norms to find that for $k \ge 1$,
\begin{equation}\label{3.22}
(\lambda_h/d)^k \|\chi_h\|_{L^{q_k}(\nu)} = (\lambda_h/d)^{k-1} \big\| \mbox{\f $\dis\frac{1}{d}$}\;L_h \chi_h \big\|_{L^{q_k}(\nu)} \stackrel{(\ref{3.14})}{\le} (\lambda_h/d)^{k-1} \|\chi_h\|_{L^{q_{k-1}}(\nu)},
\end{equation}

\n
and we obtain (\ref{3.19}) by induction (since it holds when $k = 0$).

\medskip
We now turn to the proof of (\ref{3.20}). By (\ref{3.17}) it is immediate that $h \r \lambda_h$ is a non-increasing $[0,d]$-valued function. As we now explain
\begin{equation}\label{3.23}
\mbox{$h \r \lambda_h$ is decreasing.}
\end{equation}

\n
Indeed, for $h > h'$, one has $\lambda_h = \big\langle \chi_h, L_h \, \chi_h \big\rangle_\nu = \big\langle \chi_h, L_{h'} \chi_h\big\rangle_\nu$, but $\chi_h$ is not a multiple of $\chi_{h'}$ (since $\chi_{h'}$ does not vanish on $[h',h)$), and the last expression is smaller than $\langle \chi_{h'}, L_{h'}\, \chi_{h'} \big\rangle_\nu = \lambda_{h'}$, whence (\ref{3.23}). Next, we show that
\begin{equation}\label{3.24}
\mbox{$h \r \lambda_h$ is continuous.}
\end{equation}

\n
By (\ref{3.17}) this function is a supremum of continuous functions on $\IR$, hence it is lower semi-continuous, and by (\ref{3.23}) it is right-continuous. Let us now prove that it is left-continuous. Given $h \in \IR$, since $L$ is a compact operator, we can find $h_n \uparrow h$ such that $\wt{\rho_n} = L \pi_{h_n} \chi_{h_n} = L \chi_{h_n} \r \wt{\rho}$ in $L^2(\nu)$. We set $\rho = \pi_h \,\wt{\rho}$. Then,
\begin{align}
& \pi_{h_n} \wt{\rho}_n - \rho = \pi_{h_n} (\wt{\rho}_n - \wt{\rho}) + (\pi_{h_n} - \pi_h) \,\wt{\rho} \underset{n}{\longrightarrow}  0 \; \mbox{in $L^2(\nu)$, so that} \label{3.25}
\\
&\lambda_{h_n} \chi_{h_n} = \pi_{h_n} L \pi_{h_n} \chi_{h_n} = \pi_{h_n} \wt{\rho}_n \,\underset{n}{\stackrel{L^2(\nu)}{\longrightarrow}} \rho, \;\mbox{where $\rho = 0$ on $(-\infty,h)$}. \label{3.26}
\end{align}

\n
In particular, looking at norms, one finds $\|\rho\|_{L^2(\nu)} = \lim\limits_n \lambda_{h_n} = \lambda^-_h \ge \lambda_h$ (with $\lambda^-_h$ the left-limit at $h$ of $\lambda_\point$). Moreover, one has
\begin{equation*}
L_h \rho = \pi_h L \pi_h\, \rho = \pi_h L \rho \stackrel{(\ref{3.26})}{=} \lim\limits_n \pi_h L \lambda_{h_n} \chi_{h_n} \stackrel{(\ref{3.18}), h_n < h}{=} \lim\limits_n \lambda_{h_n} \pi_h \lambda_{h_n} \chi_{h_n} \stackrel{(\ref{3.26})}{=} \lambda^-_h \pi_h \rho = \lambda^-_h \rho.
\end{equation*}

\n
Since $\rho$ vanishes on $(-\infty,h)$, this shows that $\lambda^-_h \le \lambda_h$. The left-continuity of $\lambda_\point$ and (\ref{3.24}) follow. It is now a simple fact that
\begin{equation}\label{3.27}
\lim\limits_{h \r -\infty} \lambda_h = d.
\end{equation}

\n
Indeed, $\lambda_h \le d$ and $\big\langle 1, L_h 1\big\rangle_\nu \r \big\langle 1, L1\big\rangle_\nu = d$, as $h \r -\infty$ and $\|1\|_{L^2(\nu)} = 1$. Finally, 
\begin{equation}\label{3.28}
\lim\limits_{h \r + \infty} \lambda_h = 0.
\end{equation}

\n
Indeed, by compactness of $L$, for some $h_n \uparrow \infty$, $\psi_n = L \pi_{h_n} \chi_{h_n} \underset{n}{\longrightarrow} \psi$ in $L^2(\nu)$. Then, $\lambda_{h_n} \chi_{h_n} \stackrel{(\ref{3.18})}{=} \pi_{h_n}  \psi_n = \pi_{h_n} \psi + \pi_{h_n} (\psi_n - \psi) \underset{n}{\longrightarrow} 0$ in $L^2(\nu)$. Looking at $L^2(\nu)$-norms, the claim (\ref{3.28}) follow. This completes the proof of Proposition \ref{prop3.1}. 
\end{proof}

\begin{remark}\label{rem3.2} \rm  It is not hard to infer from (\ref{3.21}) that $\chi_h$ is $C^\infty$ when restricted to $[h,+ \infty)$. However, the above proposition leaves open questions concerning the monotonicity or the boundedness of $\chi_h$, see also Remarks \ref{rem3.4} and \ref{4.4}. \hfill $\square$
\end{remark}

\medskip
We now want to characterize the critical level for level-set percolation of $\varphi$ on $T$. Level-set percolation of $\varphi$ on $T^+$, see (\ref{3.8}) for notation, amounts to the study of a specific branching Markov chain on $\IR$ with a barrier, where each individual has a level $a \ge h$, and independently gives rise in the next generation to $d$ individuals with levels distributed as independent $N(\frac{a}{d}, \frac{1}{d}$)-variables, and these individuals are killed if the level falls below $h$. We refer to \cite{BiggKypr04} and the references therein for related models. With this in mind, for $h \in \IR$ and $n \ge 0$ (see (\ref{3.9}) and below (\ref{3.4}) for notation), we define
\begin{equation}\label{3.29}
\cZ^h_n = \{x \in T^+_n; \; \varphi_y \ge h \; \mbox{for all} \; y \in [x_0,x]\},
\end{equation}

\n
so that $\cZ^h_n$ denotes the intersection of $T^+_n$ with the cluster of $x_0$ in the level-set $\{\varphi \ge h\}$. We also introduce the filtration
\begin{equation}\label{3.30}
\cF_n = \sigma(\varphi_x; x \in T^+, d(x_0,x) \le n), \, n \ge 0,
\end{equation}
as well as the $(\cF_n)$-adapted process
\begin{equation}\label{3.31}
M_n = \lambda_h^{-n} \dsl_{x \in \cZ^h_n} \chi_h(\varphi_x), \; n \ge 0.
\end{equation}

\begin{proposition}\label{prop3.3}
Define $h_* \in \IR$ as the unique value (cf.~(\ref{3.20})) such that
\begin{equation}\label{3.32}
\lambda_{h_*} = 1.
\end{equation}

\vspace{-5ex}
\begin{align}
\mbox{For $h < h_*$, $\IP^G$-a.s., $\{\varphi \ge h\}$} &\; \mbox{has an infinite component}. \label{3.33}
\\[1ex]
\mbox{For $h > h_*$, $\IP^G$-a.s., $\{\varphi \ge h\}$} &\; \mbox{only has finite components}. \label{3.34}
\end{align}
\end{proposition}

\begin{proof}
We first prove (\ref{3.33}). Note that classically, see Theorem 1, p.~247 in Chapter 6 \S 4 \cite{AthrNey04}, one knows (with the help of (\ref{3.13})) that
\begin{equation}\label{3.35}
\mbox{$(M_n)_{n \ge 0}$ is a non-negative $(\cF_n)$-martingale under $\IP^G$}.
\end{equation}

\medskip\n
We will show that when $h< h_*$, $M_n$ converges a.s. to a non-identically vanishing limit. We could base the proof on (2.1) of Theorem 2.1 in \cite{BiggKypr04}, but it is actually slightly more precise and essentially as quick to observe that under $Q = \chi_h (\varphi_{x_0}) / \big\langle \chi_h\big\rangle_\nu \IP^G$ (where $\big\langle f \big\rangle_\nu$ stands for $\int_\IR f(d\nu)$), $(M_n)_{n \ge 0}$ is an $(\cF_n)$-martingale (note that $\frac{dQ}{d \IP^G}$ is $\cF_0$-measurable) and
\begin{equation}\label{3.36}
\sup\limits_{n \ge 0} E^Q [ M_n^2] < \infty.
\end{equation}
Indeed, for $n \ge 1$, by the orthogonality of the increments
\begin{equation}\label{3.37}
\begin{array}{l}
E^Q[M^2_n] = E^Q[M^2_0] + \dsl^n_{k=1} \,E^Q[(M_k - M_{k-1})^2], \;\mbox{and}
\\[1ex]
M_k - M_{k-1} = \lambda_h^{-(k-1)} \dsl_{x \in \cZ^h_{k-1}} \; \dsl^d_{j=1} \,\big(\lambda^{-1}_h \chi_h (\varphi_{(x,j)}) - \mbox{\f $\dis\frac{1}{d}$} \;\chi_h(\varphi_x)\big),
\end{array}
\end{equation}

\n
where $(x,1),\dots,(x,d)$ denote the descendants of $x \in T^+$.

\medskip
Note that (\ref{3.19}) ensures the finiteness of the above expectations. By (\ref{3.7}), (\ref{3.13}) (applied at neighboring sites in $T^+_k$ and $T^+_{k-1}$) the summands under parenthesis in the second line of (\ref{3.37}), conditionally on $\cF_{k-1}$, are centered and independent under $\IP^G$ or $Q$. We thus find that for $k \ge 1$,
\begin{equation}\label{3.38}
\begin{split}
E^Q[(M_k - M_{k-1})^2] & = \lambda^{-2k}_h \,E^Q\Big[\dsl_{x \in \cZ^h_{k-1}} \;\dsl^d_{j=1} \,E^Q[\big(\chi_h(\varphi_{(x,j)}) - \mbox{\f $\dis\frac{\lambda_h}{d}$} \; \chi_h(\varphi_x)\big)^2 | \cF_{k-1}]\Big]
\\
& = \lambda_h^{-2k} \,E^Q \Big[\dsl_{x \in \cZ_{k-1}^h} \,d\big(Q_{t_1} \chi^2_h - \mbox{\f $\dis\frac{\lambda^2_h}{d^2}$} \,\chi^2_h\big) (\varphi_x)\Big],
\end{split}
\end{equation}

\medskip\n
using the conditional centering, (\ref{3.13}) and (\ref{3.18}) in the last step. Note that for $x \in T^+_{k-1}$
\begin{equation}\label{3.39}
\begin{array}{l}
E^Q[x \in \cZ^h_{k-1}, f(\varphi_x)]  = \IE^G[\chi_h(\varphi_{x_0}), \,\varphi_y \ge h \;\mbox{for all $y \in [x_0,x], \,f(\varphi_x)] / \big\langle \chi_h\big\rangle_\nu$}
\\
= \big(\mbox{\f $\dis\frac{\lambda_h}{d}$}\big)^{k-1} \IE^G[\chi_h(\varphi_x) \,f(\varphi_x)] / \big\langle \chi_h\big\rangle_\nu, \;\mbox{for $f \ge 0$ measurable on $\IR$},
\end{array}
\end{equation}

\n
where we have made iterated use of the Markov property and (\ref{3.18}). Since $\chi_h \in L^3(\nu)$ by (\ref{3.19}), we can replace $f$ by $Q_{t_1} \chi^2_h - \frac{\lambda_h^2}{d^2} \;\chi^2_h$ in the above, and sum over $x \in T^+_{k-1}$ to find
\begin{equation}\label{3.40}
\begin{split}
E^Q[(M_k - M_{k-1})^2] & = d^{k-1} d \,\lambda_h^{-2k} \;\big(\mbox{\f $\dis\frac{\lambda_h}{d}$}\big)^{k-1} \big\langle\chi_h, Q_{t_1}(\chi^2_h) - \mbox{\f $\dis\frac{\lambda_h^2}{d^2}$} \;\chi^2_h\big\rangle_\nu / \big\langle \chi_h\big\rangle_\nu
\\
& = d \lambda_h^{-(k+1)} \big\langle \mbox{\f $\dis\frac{\lambda_h}{d}$} \;\chi^3_h - \mbox{\f $\dis\frac{\lambda_h^2}{d^2}$} \;\chi^3_h \big\rangle_\nu / \big\langle \chi_h\big\rangle_\nu = \lambda^{-k}_h \;\mbox{\f $\dis\frac{\big\langle\chi^3_\nu\big\rangle_\nu}{\big\langle\chi_h\big\rangle_\nu}$} \;\big(1 - \mbox{\f $\dis\frac{\lambda_h}{d}$}\big).
\end{split}
\end{equation}
Coming back to (\ref{3.37}) we find that
\begin{equation}\label{3.41}
E^Q[M^2_n] = \Big(1 + \dsl^n_{k=1} \lambda^{-k}_h \big(1 - \mbox{\f $\dis\frac{\lambda_h}{d}$}\big)\Big) \big\langle \chi^3_h\big\rangle_\nu / \big\langle \chi_h\big\rangle_\nu ,
\end{equation}

\medskip\n
and (\ref{3.36}) follows. In particular, $M_n \r M_\infty$ $Q$-a.s. and hence $\IP^G$-a.s., where $M_\infty \ge 0$ and $E^Q[M_\infty] = E^Q[M_0] = \big\langle \chi^2_h\big\rangle_\nu / \big\langle \chi_h\big\rangle_\nu = \big\langle \chi_h\big\rangle_\nu^{-1} > 0$. This implies that with $\IP^G$-positive measure $\cZ^h_n$ is non-empty for all $n \ge 0$, and the connected component of $x_0$ in $\{\varphi \ge h\} \cap T^+$ is infinite. The claim (\ref{3.33}) can for instance be deduced from the weak law of large numbers in Corollary 10 of \cite{Pema92}. 

\medskip
We will now prove (\ref{3.34}). We consider the function in $L^\infty(\nu)$
\begin{equation}\label{3.42}
q_h(a) = \IP^G[|\cZ^h_n| = 0, \;\mbox{for large $n| \varphi_{x_0} = a], \, a \in \IR$},
\end{equation}

\n
where we make use of the Markov property (\ref{3.6}) (with $x'$ chosen as $x_0$) to define the conditional expectation in (\ref{3.42}), and $|A|$ stands for the cardinality of $A$ when $A \subseteq T$. We will show that
\begin{equation}\label{3.43}
\mbox{when $h > h_*$, $\nu$-a.s., $q_h = 1$},
\end{equation}

\n
and quickly deduce (\ref{3.34}). We first introduce for general $h \in \IR$, $n \ge 0$, $1 \le i \le d$, 
\begin{equation*}
\mbox{$\cZ^h_n(i) = \{x \in T^+_{n+1}$; $x$ is a descendant of $(x_0,i)$ and $\varphi_y \ge h$, for all $y \in [(x_0,i),x]\}$. }
\end{equation*}

\n
The Markov property (\ref{3.6}) implies that $\IP^G$-a.s.
\begin{equation}\label{3.44}
\IP^G\big[\mbox{\f \small $\bigcap\limits_{i=1}^d$} \{| \cZ^h_n(i)| = 0, \;\mbox{for large} \; n\} |\varphi_{x_0}, \varphi_{(x_0,1)},\dots,\varphi_{(x_0,d)}\big] = \mbox{\f $\prod\limits^d_{i=1}$} \;q_h (\varphi_{(x_0,i)}),
\end{equation}
and hence $\IP^G$-a.s.
\begin{equation*}
\begin{split}
q_h(\varphi_{x_0}) & = 1_{(-\infty,h)}(\varphi_{x_0}) + \IE^G \big[\varphi_{x_0} \ge h, \mbox{\small $\bigcap^d_{i=1}$} \; \{| \cZ^h_n(i)| = 0, \; \mbox{for large} \; n\}| \varphi_{x_0}\big]
\\[-0.5ex]
&\!\!\! \stackrel{(\ref{3.44})}{=} 1_{(-\infty,h)}(\varphi_{x_0}) + \IE^G \big[\varphi_{x_0} \ge h, \; \mbox{\small $\prod\limits^d_{i=1}$} \; q_h (\varphi_{(x_0,i)}) | \varphi_{x_0}\big]
\\[-0.5ex]
&\!\!\! \stackrel{(\ref{3.13})}{=} 1_{(-\infty,h)}(\varphi_{x_0}) + Q_{t_1} (q_h)^d (\varphi_{x_0}) \, 1_{[h,\infty)}   (\varphi_{x_0}), \;\mbox{so that}
\end{split}
\end{equation*}
\begin{equation}\label{3.45}
q_h = 1_{(-\infty,h)} + 1_{[h,\infty)} \big(Q_{t_1} (q_h)\big)^d, \;\mbox{$\nu$-a.s.}\,.
\end{equation}

\medskip\n
Thus, setting $r_h = 1 - q_h$, we find that $\nu$-a.s.,
\begin{equation}\label{3.46}
\begin{split}
r_h & = 0 \; \mbox{on $(-\infty,h)$}
\\[-1ex]
& = 1 - Q_{t_1} (q_h)^d = Q_{t_1} (r_h) \;\dsl^{d-1}_{k=0} \, Q_{t_1} (q_h)^k \le d Q_{t_1}(r_h) = L_h (r_h) \; \mbox{on $[h,\infty)$}.
\end{split}
\end{equation}
This shows that $\nu$-a.s.,
\begin{equation}\label{3.47}
0 \le r_h \le L_h \, r_h,
\end{equation}

\n
and if $r_h$ is not a.s. $0$, then by (\ref{3.17}), $\lambda_h \ge 1$. But $h > h_*$ implies $\lambda_h < 1$, whence (\ref{3.43}). As a result, $\IP^G$-a.s., the cluster of $x_0$ in $T^+ \cap \{\varphi \ge h\}$ is finite, and (\ref{3.34}) easily follows.
\end{proof}

\begin{remark}\label{rem3.4} \rm
In proving (\ref{3.34}) one might try to use iii) of Theorem 2.1. of \cite{BiggKypr04} to argue that $\IP^G$-a.s., $M_\infty = 0$, when $h > h_*$, but our lack of understanding of the asymptotic behavior of $\chi_h$ (see Remark \ref{rem3.2}) does not seem to then lead to a quick conclusion that a.s. $|\cZ^h_n| = 0$ for large $n$. For this reason we used the above argument. \hfill $\square$
\end{remark}

\section{Bounds on $\lambda_h$ and $h_*$}
\setcounter{equation}{0}

In this section we derive upper and lower bounds on the critical value $h_*$ for the level-set percolation of the Gaussian free field $\varphi$ on the $(d+1)$-regular tree $T$, with $d \ge 2$. These results appear in Corollary \ref{cor4.5} and are direct consequences of the lower and upper bounds on $\lambda_h$ from Theorem \ref{theo4.3}. The upper bound comes as an application of the coupling between the Gaussian free field and random interlacements from Corollary \ref{cor2.4} in \hbox{Section 2.}

\medskip
We keep the notation of the previous sections. We let $\wt{T}$ stand for the metric graph attached to the $(d+1)$-regular tree endowed with unit weights. As in Sections 1 and 2, $(\wtv_z)_{z \in \wt{T}}$ stands for the canonical Gaussian free field on $\wt{T}$ and is governed by the probability $\wt{\IP}^G$. In the present context assumption (\ref{1.33}) is automatic, see (\ref{3.3}), and we will now prove that (\ref{1.31}) holds as well.

\begin{proposition}\label{prop4.1}
\begin{equation}\label{4.1}
\mbox{$\wt{\IP}^G$-a.s., $\{\wtv > 0\}$ only has bounded components on $\wt{T}$}.
\end{equation}
\end{proposition}

\begin{proof}
Given $x,y \in T$, we write $[\wt{x,y}] \subseteq \wt{T}$ for the geodesic segment in $\wt{T}$ between $x$ and $y$. For convenience we write $\varphi_x$ in place of $\wtv_x$ when $x \in T\,(\subseteq \wt{T})$ in the proof of Proposition \ref{prop4.1}. With $T^+$ and $T^+_n$ as in (\ref{3.8}), (\ref{3.9}), we introduce for $n \ge 0$,
\begin{equation}\label{4.2}
\wt{\cZ}_n = \{x \in T^+_n; \;\wtv > 0 \;\mbox{on} \; [\wt{x_0,x}]\}.
\end{equation}
We will show that (with similar notation as below (\ref{3.42}))
\begin{equation}\label{4.3}
\mbox{$\wt{\IP}^G$-a.s., \; $| \wt{\cZ}_n | = 0$ for large $n$,}
\end{equation}

\n
and (\ref{4.1}) will quickly follow. By Lemma 3.1 of \cite{Lupu} and Lemma 10.12, p.~145 of \cite{Jans97} or Proposition 5.2 of \cite{Lupu}, one knows that for any $x \in T^+_n$,
\begin{equation}\label{4.4}
\begin{array}{l}
\mbox{$\wt{\IP}^G\big[\wtv$ does not vanish on $[\wt{x_0,x}] \big] = \wtie^G[{\rm sign}(\varphi_{x_0}) \, {\rm sign}(\varphi_x)] =$}
\\[1ex]
\mbox{\f $\dis\frac{2}{\pi}$} \;\arcsin \Big(\mbox{\f $\dis\frac{g(x_0,x)}{\sigma^2}$}\Big) \stackrel{(\ref{3.1}) - (\ref{3.3})}{=} \mbox{\f $\dis\frac{2}{\pi}$} \arcsin \Big( \mbox{\f $\dis\frac{1}{d^n}\Big)$} .
\end{array}
\end{equation}
In particular, we see that
\begin{equation}\label{4.5}
\mbox{$\wt{\IE}^G [| \wt{\cZ}_n |] \le \mbox{\f $\dis\frac{2}{\pi}$} \;d^n \arcsin  \Big( \mbox{\f $\dis\frac{1}{d^n}$}\Big)$, for $n \ge 0$}.
\end{equation}
We will use the following

\begin{lemma}\label{lem4.2}
\begin{equation}\label{4.6}
\begin{split}
\wt{M}_n & = \mbox{$\dsl_{x \in \wt{\cZ}_n} \varphi_x$, $n \ge 0$, is a non-negative $(\wt{\cF}_n)$-martingale, where}
\\
\wt{\cF}_n & = \sigma \big(\wtv_z, \,z \in \mbox{\f \small $\bigcup\limits_{x \in T^+_n}$} \;(\wt{[x_0,x]})\big).
\end{split}
\end{equation}
\end{lemma}

\begin{proof}
Note that $\wt{M}_n$ is clearly non-negative, $\wt{\cF}_n$-adapted, integrable. Moreover, for $n \ge 0$,
\begin{equation}\label{4.7}
\wtie^G [\wt{M}_{n+1} \,| \,\wt{\cF}_n] = \dsl_{x \in \wt{\cZ}_n} \;\dsl^d_{j=1} \;\wtie^G\big[\varphi_{(x,j)}, \wtv > 0 \; \mbox{on} \; [\wt{x,(x,j)}]\, |\,  \wt{\cF}_n\big].
\end{equation}

\n
We will now compute the conditional expectation in the right-hand side. Consider fixed $x \in T^+_n$, a fixed $j \in \{1,\dots,d\}$, and define with hopefuly obvious notation
\begin{align}
\cH = \big(\mbox{\small $\bigcup\limits_{y \in T^+_n}$} [\wt{x_0,y}]\big)\; \cup &\; \mbox{$\big($the closure of the connected component of} \label{4.8}
\\[-2ex]
&\;\; \{\wtv > 0\} \cap \big[x,(x,j)\big)\big). \nonumber
\end{align}

\medskip\n
Then, $\cH$ is a compatible random compact subset of $\wt{T}$, see above (\ref{1.10}) (actually, as a minor point, when $n = 0$, we replace $x_0$ by $x_{-1}$ in (\ref{4.8}) to ensure that $\cH$ is the closure of its interior). Since $\cH$ contains $\bigcup_{y \in T^+_n} [\wt{x_0,y}]$, it follows from (\ref{1.12}) iii) that $\wt{\cF}_n \subseteq \cA^+_\cH$. By the strong Markov property (\ref{1.18}), we have
\begin{equation}\label{4.9}
\wtie^G [\varphi_{(x,j)} \, | \, \cA^+_\cH] = \wte_{(x,j)} [\wtv(X_{H_{\cH}}), H_\cH < \infty] = 0 \;\mbox{on} \; \{(x,j) \notin \cH, x \in \wt{\cZ}_n\},
\end{equation}
since when $x \in \wt{\cZ}_n$ and $(x,j) \notin \cH$, $\wt{P}_{(x,j)}$-.a.s., $\wtv(X_{H_\cH}) = 0$ on $\{H_\cH < \infty\}$. So $\wt{\IP}^G$-a.s.,
\begin{equation}\label{4.10}
\begin{array}{l}
\wt{\IE}^G[ \varphi_{(x,j)} \,|\, \wt{\cF}_n ] 1 \{x \in \wt{\cZ}_n\} = \wt{\IE}^G[\wt{\IE}^G[ \varphi_{(x,j)}\, \, | \, \cA^+_\cH] \,| \, \wt{\cF}_n] \; 1\{x \in \wt{\cZ}_n\} \stackrel{(\ref{4.9}),(\ref{1.12})\,{\rm ii)}}{=}
\\[1ex]
\wt{\IE}^G[ \varphi_{(x,j)} \,1 \{(x,j) \in \cH\} \, | \,  \wt{\cF}_n] \,1\{x \in  \wt{\cZ}_n\} = 
\\[1ex]
\wtie^G[ \varphi_{(x,j)}, \wtv > 0 \;\; \mbox{on}  \;  [\wt{x,(x,j)}] \, | \,  \wt{\cF}_n]  \, 1\{x \in \wt{\cZ}_n\}.
\end{array}
\end{equation}

\n
The last expression is precisely the summand in the right-hand side of (\ref{4.7}). On the other hand, by the Markov property for $\wtv$, see (\ref{1.8}), we have
\begin{equation}\label{4.11}
\wtie^G [ \varphi_{(x,j)}  \, | \,  \wt{\cF}_n] = \mbox{\f $\dis\frac{1}{d}$} \; \varphi_x.
\end{equation}

\n
Inserting this equality in the left-hand side of (\ref{4.10}) and coming back to (\ref{4.7}), yields that $\wtie^G [\wt{M}_{n+1}  \, | \,  \wt{\cF}_n]  = \sum_{x \in \wt{\cZ}_n} d\times \frac{1}{d} \, \varphi_x = \wt{M}_n$. This proves (\ref{4.6}).
\end{proof}

\medskip
By the martingale convergence theorem, we find that
\begin{equation}\label{4.12}
\mbox{$\wt{\IP}^G$-a.s., $\wt{M}_n \longrightarrow \wt{M}_\infty \ge 0$,  where $\wtie^G [\wt{M}_\infty ] < \infty$}.
\end{equation}
As we now explain
\begin{equation}\label{4.13}
\mbox{$\wt{\IP}^G$-a.s., $| \wt{\cZ}_n | = 0$ for large $n$}.
\end{equation}

\n
Indeed, set for $M \ge 1$ and $n \ge 1$, $A_{M,n} = \{\sum_{x \in \wt{\cZ}_n} (\varphi_x + 1) \le M\}$ as well as $A_M = \limsup_n A_{M,n}$. By the Markov property of $\wtv$, see (\ref{1.8}), on $A_{M,n}$, we have
\begin{equation*}
\begin{split}
\wt{\IP}^G [ |\, \wt{\cZ}_{n+1} | = 0 \,| \, \wt{\cF}_n]  & \ge \wt{\IP}^G [\varphi_{(x,j)} < 0, \;\mbox{for all} \; x \in \wt{\cZ}_n, j = 1,\dots, d  \,| \, \wt{\cF}_n] 
\\
& =\mbox{\small $ \prod\limits_{x \in \wt{\cZ}_n}$} Q_{t_1} (1_{(-\infty,0)})^d (\varphi_x)
\\
&\ge Q_{t_1} (1_{(-\infty,0)})^{dM} (M) \stackrel{\rm def}{=} c(d,M), \;\mbox{by definition of $A_{M,n}$}.
\end{split}
\end{equation*}

\medskip\n
It now follows from Borell-Cantelli's lemma that
\begin{equation}\label{4.14}
\mbox{$\wt{\IP}^G$-a.s., on $A_M,  |\, \wt{\cZ}_n | = 0$ for large $n$}.
\end{equation}

\medskip\n
By (\ref{4.5}) and Fatou's lemma $\wtie^G[\liminf\limits_n \,|\wt{\cZ}_n|] \le \frac{2}{\pi}$. Combined with (\ref{4.12}), we see that $\wt{\IP}^G[\bigcup_{M \ge 1} A_M] = 1$ and (\ref{4.13}) follows. Hence, $\wt{\IP}^G$-a.s., the connected component of $x_0$ in $\{\wtv > 0\} \cap (\bigcup_{x \in T^+} [\wt{x_0,x}])$ is bounded and (\ref{4.1}) readily follows.
\end{proof}

\medskip
We now come to the main estimates on the quantity $\lambda_h$ from (\ref{3.17}). We recall the notation $\overline{\Phi}\,(a) = \frac{1}{\sqrt{2\pi}} \int^{+\infty}_a\, e^{-\frac{s^2}{2}} ds$ from (\ref{0.13}). The upper bound will crucially rely on Corollary \ref{cor2.4}.

\begin{theorem}\label{theo4.3} $(d \ge 2)$
\begin{align}
\mbox{For all $h \in \IR$}, & \;\lambda_h > d \overline{\Phi} \,\big(h \mbox{\f $\dis\frac{(d-1)}{\sqrt{d}}$} \big) \; \mbox{$\big($in particular $\lambda_0 >\frac{d}{2}\big)$}.  \label{4.15}
\\[1ex]
\mbox{For all $h \ge 0$}, & \;\lambda_h \le \lambda_0  \;e^{-\frac{h^2(d-1)^2}{2d}} \;   \big(< d \,e^{-\frac{h^2(d-1)^2}{2d}} \big).  \label{4.16}
\end{align}

\n
(We recall that $\frac{1}{\sqrt{2 \pi}} \; (a + \frac{1}{a})^{-1} \, e^{-\frac{a^2}{2}} \le \overline{\Phi}(a) \le \frac{1}{\sqrt{2 \pi}a} \; e^{-\frac{a^2}{2}}$, for $a > 0$).
\end{theorem}

\begin{proof}
We begin by the proof of (\ref{4.15}). For $h \in \IR$ and $a > h$, one has, cf.~(\ref{3.15}),
\begin{equation}\label{4.17}
\begin{split}
L_h \, 1_{[h,+ \infty)} (a) & = d \,P^\xi \big[ \mbox{\f $\dis\frac{a}{d}$} + \xi \ge h\big] \;\;\mbox{with $\xi$ a $N\big(0,  \mbox{\f $\dis\frac{1}{d}$}\big)$-distributed variable}
\\[1ex]
& \stackrel{a > h}{>} d \,P^\xi \big[\xi > h  \mbox{\f $\dis\frac{(d-1)}{d}$}\big] = d\,\overline{\Phi} \;\big( h  \mbox{\f $\dis\frac{(d-1)}{\sqrt{d}}$}\big). 
\end{split}
\end{equation}
As a result, we find that $\nu$-a.s.~on $[h,\infty)$,
\begin{equation}\label{4.18}
L_h \, 1_{[h,\infty)}  > d \, \overline{\Phi} \;\big( h \mbox{\f $\dis\frac{(d-1)}{\sqrt{d}}$}\big) ,
\end{equation}

\n
and (\ref{4.15}) readily follows from the variational formula (\ref{3.17}). 

\medskip
We now turn to the proof of (\ref{4.16}). For $h < h'$ we define $\pi_{h,h'}$ as the multiplication operator by $1_{[h,h']}$ in $L^2(\nu)$, and set
\begin{equation}\label{4.19}
L_{h,h'} = \pi_{h,h'} \, L \pi_{h,h'}, \;\mbox{with $L$ as in (\ref{3.15})}.
\end{equation}

\n
Then, $L_{h,h'}$ is a self-adjoint, non-negative, Hilbert-Schmidt operator on $L^2(\nu)$, and we introduce the maximum eigenvalue
\begin{equation}\label{4.20}
\lambda_{h,h'} = \| L_{h,h'}\|_{L^2(\nu) \r L^2(\nu)} \le \lambda_h \;\; \mbox{(see (\ref{3.17}))}.
\end{equation}

\n
The same proof as in Proposition \ref{prop3.1} shows that $\lambda_{h,h'}$ is a simple eigenvalue of $L_{h,h'}$ and we have a uniquely defined function
\begin{equation}\label{4.21}
\begin{array}{l}
\mbox{$\chi_{h,h'}$ non-negative, vanishing outside $[h,h']$, continuous and strictly positive}
\\
\mbox{on $[h,h']$, with unit $L^2(\nu)$-norm, eigenfunction of $L_{h,h'}$ for the eigenvalue $\lambda_{h,h'}$}.
\end{array}
\end{equation}

\n
It is also immediate that $\lambda_{h,h'}$ is non-decreasing in $h'(> h)$ and since $\lambda_h \ge \lambda_{h,h'} \ge \big\langle L_{h,h'} \, \chi_h, \, \chi_h \big\rangle \longrightarrow \lambda_h$, as $h' \r \infty$, one has
\begin{equation}\label{4.22}
\lim\limits_{h' \r \infty} \,\lambda_{h,h'} = \lambda_h .
\end{equation}
In analogy with (\ref{3.29}) we introduce
\begin{equation}\label{4.23}
\cZ_n^{h,h'} = \{x \in T^+_n; \, \varphi_y \in [h,h'] \; \mbox{for all $y \in [x_0,x]\}$},
\end{equation}

\n
and note that for all $x \in T^+_n$, by iterated application of (\ref{3.13})
\begin{equation}\label{4.24}
\begin{split}
\IE^G [ \chi_{h,h'} (\varphi_{x_0}), \, x \in\cZ_n^{h,h'}] & = \big\langle \chi_{h,h'}, \big(\mbox{\f $\dis\frac{1}{d}$} \; L_{h,h'}\big)^n 1 \big\rangle_\nu
\\ 
&\!\!\! \stackrel{(\ref{4.21})}{=} (\lambda_{h,h'} / d)^n \, \big\langle \chi_{h,h'}\big\rangle_\nu ,
\end{split}
\end{equation}

\n
where $\langle \cdot \rangle_\nu$ stands for $\nu$-expectation (as below (\ref{3.35})). In particular, summing over $x$ in $T^+_n$, one obtains
\begin{equation}\label{4.25}
\lambda^n_{h,h'} \, \big\langle \chi_{h,h'}\big\rangle_\nu = \IE^G[\chi_{h,h'} (\varphi_{x_0}) \, | \cZ_n^{h,h'} | ] , \; \mbox{for $n \ge 0$, $h < h'$}.
\end{equation}

\n
We now assume $0 \le h \le h'$, and note that by (\ref{2.47}) of Corollary \ref{cor2.4}, one has for $x \in T^+_n$
\begin{equation}\label{4.26}
\IP^G [x \in \cZ_n^{h,h'}] \le \IP^G [ x \in \cZ_n^{0,h'-h}] \,\IP^I\big[[x_0,x] \subseteq \cV^u\big], \; \mbox{with $h = \sqrt{2u}$}.
\end{equation}

\n
By (5.9) - (5.9) of \cite{Teix09b} (note that $d$ in \cite{Teix09b} corresponds to $d+1$ here), one knows that
\begin{equation}\label{4.27}
\begin{split}
\IP^I \big[[x_0,x] \subseteq \cV^u\big] & = \exp\big\{- u \,({\rm cap}\{x_0\}) - n \,u \mbox{\f $\dis\frac{(d-1)^2}{d}$}\big\}
\\ 
&\!\! \stackrel{(\ref{3.3})}{=} \exp\big\{ -  \mbox{\f $\dis\frac{u}{\sigma^2}$} - n \,u  \mbox{\f $\dis\frac{(d-1)^2}{d}$}\big) = \exp\big\{ - u \big(d - \mbox{\f $\dis\frac{1}{d}$}\big)  - n \, u  \mbox{\f $\dis\frac{(d-1)^2}{d}$}\big\}.
\end{split}
\end{equation}

\n
Summing over $x \in T^+_n$ in (\ref{4.26}) we thus find that for $n \ge 0$, $0 \le h < h'$,
\begin{equation}\label{4.28}
\IE^G [| \cZ_n^{h,h'}|] \le \IE^G [ | \cZ_n^{0,h'-h} | ] \; \exp \big\{ - u \big(d - \mbox{\f $\dis\frac{1}{d}$}\big)  - n \, u  \mbox{\f $\dis\frac{(d-1)^2}{d}$}\big\}.
\end{equation}

\n
By (\ref{4.25}), the left-hand side of (\ref{4.28}) is at least
\begin{equation*}
\lambda^n_{h,h'} \, \big\langle\chi_{h, h'}\big\rangle_\nu / \|\chi_{h,h'}\|_\infty,
\end{equation*}

\n
and the expection in the right-hand side of (\ref{4.28}) is at most
\begin{equation*}
\lambda^n_{0, h'-h}\ \, \big\langle\chi_{0, h'-h}\big\rangle_\nu / \inf\limits_{[0,h'-h]} \chi_{0,h'-h}.
\end{equation*}

\n
Taking $n$-th roots and letting $n$ go to infinity, we find that
\begin{equation}\label{4.29}
\lambda_{h,h'} \le \lambda_{0,h'-h} \,e^{-\frac{u(d-1)^2}{d}}, \;\mbox{for $0 \le h < h'$}.
\end{equation}
Letting $h' \uparrow \infty$ yields (\ref{4.16}) (recall that $h = \sqrt{2u}$).
\end{proof}

\begin{remark}\label{rem4.4} \rm 
We do not know of a derivation of the upper bound (\ref{4.16}) directly from the variational formula (\ref{3.17}) (or (\ref{0.9})). Incidentally, the above proof makes use below (\ref{4.28}) of quantities such as $\| \chi_{h,h'}\|_\infty$ and $(\inf_{[0,h'-h]} \chi_{h,h'})^{-1}$. We have no bounds on such quantities when $\chi_h$ replaces $\chi_{h,h'}$, see Remark \ref{rem3.2}. The approximations $\lambda_{h,h'}$ of $\lambda_h$ in the above proof enable us to bypass this lack of controls. \hfill $\square$
\end{remark}

With the help of Proposition \ref{prop3.3}, the above Theorem \ref{theo4.3} yields bounds on $h_*$.

\begin{corollary}\label{cor4.5} $(d \ge 2)$
\begin{equation}\label{4.30}
0 \le h_\Delta < h_* \le h_{\square} < \sqrt{2u_*},
\end{equation}
where $h_\Delta, h_\square$ are defined by
\begin{equation}\label{4.31}
d \overline{\Phi}     \, \big(h_\Delta \,\mbox{\f $\dis\frac{(d-1)}{\sqrt{d}}$}\big) = 1, \; \lambda_0 \,e^{-\frac{h^2_\square(d-1)^2}{2d}} = 1,
\end{equation}

\n
and $u_*$ is the critical level for the percolation of $\cV^u$ on $T$ (by (5.5) of \cite{Teix09b}, $d \,e^{- u_* \,\frac{(d-1)^2}{d}} = 1$).
\end{corollary}

\begin{proof}
This is a direct consequence of (\ref{4.15}), (\ref{4.16}) and $\lambda_{h_*} = 1$ from Proposition \ref{prop3.3}, as well as of (5.5) of \cite{Teix09b}.
\end{proof}

\begin{remark}\label{rem4.6} \rm 
We thus have $0 < h_* < \sqrt{2u_*}$ for all $d \ge 2$. In the case of level-set percolation of the Gaussian free field on $\IZ^d$, $d \ge 3$, $h_* \ge 0$, is known since \cite{BricLeboMaes87}, but $h_* > 0$ is presently only known for large $d$, see \cite{DrewRodr15}, \cite{RodrSzni13}, however expected for all $d \ge 3$. The inequality $h_* \le \sqrt{2u_*}$ was proved in \cite{Lupu}. The coupling of Section 2 that we used here, may perhaps be helpful in the case of $\IZ^d$, $d \ge 3$, to show that in fact $h_* < \sqrt{2u_*}$ holds as well. 

\hfill $\square$
\end{remark}


\end{document}